\documentclass[11pt,a4paper]{article}
\usepackage[a4paper]{geometry}
\usepackage{amssymb,latexsym,amsmath,amsfonts,amsthm}
\usepackage{graphicx}

\usepackage{epsfig}
\usepackage{comment}
\usepackage{subfigure}
\usepackage{graphicx,ebezier}
\usepackage{overpic}

\DeclareMathOperator{\diag}{diag} 

\newcommand{\er}{\mathbb{R}}
\newcommand{\cee}{\mathbb{C}}

\newcommand{\zet}{\mathbb{Z}}

\newcommand{\Lam}{\Lambda}

\newcommand{\bol}{\hfill\square\\}
\newcommand{\til}{\tilde}

\newcommand{\vecv}{\mathbf{v}}
\newcommand{\vecw}{\mathbf{w}}

\newcommand{\tr}{\mathrm{Tr}}

\newcommand{\ep}{\epsilon}
\newcommand{\ud}{\mathrm{d}}

\newtheorem{theorem}{Theorem}[section]
\newtheorem{lemma}[theorem]{Lemma}
\newtheorem{proposition}[theorem]{Proposition}

\newtheorem{corollary}[theorem]{Corollary}

\theoremstyle{definition}

\theoremstyle{remark}

\numberwithin{equation}{section}

\title{Zeros and ratio asymptotics for matrix orthogonal polynomials}

\author{Steven Delvaux\footnotemark[1],\quad Holger Dette\footnotemark[2]}

\date{}

\begin{document}
\maketitle
\renewcommand{\thefootnote}{\fnsymbol{footnote}}
\footnotetext[1]{Department of Mathematics, University of Leuven (KU Leuven),
Celestijnenlaan 200B, B-3001 Leuven, Belgium. email:
steven.delvaux\symbol{'100}wis.kuleuven.be. } \footnotetext[2]{Department of
Mathematics, Ruhr-Universit\"at Bochum, 44780 Bochum, Germany. e-mail:
holger.dette\symbol{'100}rub.de. }

\begin{abstract}
Ratio asymptotics for matrix orthogonal polynomials with recurrence
coefficients $A_n$ and $B_n$ having limits $A$ and $B$ respectively (the matrix
Nevai class) were obtained by Dur\'an. In the present paper we obtain an
alternative description of the limiting ratio. We generalize it to recurrence
coefficients which are asymptotically periodic with higher periodicity, and/or
which are slowly varying in function of a parameter. Under such assumptions, we
also find the limiting zero distribution of the matrix orthogonal polynomials,
generalizing results by Dur\'an-L\'opez-Saff and Dette-Reuther to the
non-Hermitian case. Our proofs are based on \lq normal family\rq\ arguments and
on the solution to a quadratic eigenvalue problem. As an application of our
results we obtain new explicit formulas for {the spectral measures of} the
matrix Chebyshev polynomials of the first and second kind, and we derive the
asymptotic eigenvalue distribution for a class of random band matrices
generalizing the tridiagonal matrices {introduced by} Dumitriu-Edelman.


\end{abstract}

\setcounter{tocdepth}{1} \tableofcontents

\section{Introduction}
\label{section:intro}

Let $(P_n(x))_{n=0}^{\infty}$ be a sequence of matrix-valued polynomials of
size $r\times r$ ($r\geq 1$) generated by the recurrence relation
\begin{equation}\label{recurrence}
xP_n(x) = A_{n+1} P_{n+1}(x)+B_n P_n(x)+A_{n}^* P_{n-1}(x),\qquad n\geq 0,
\end{equation}
with initial conditions $P_{-1}(x)\equiv 0 \in \mathbb{C}^{r \times r}$ and $P_0(x)\equiv I_r,$ with $I_r$
the identity matrix of size $r$. The coefficients $A_k$ and $B_k$ are complex matrices
of size $r\times r$. We assume that each matrix $A_k$ is nonsingular and $B_k$ is
Hermitian.
The star superscript denotes the Hermitian conjugation.

The polynomials generated by \eqref{recurrence} satisfy orthogonality relations
with respect to a matrix-valued measure (spectral measure) on the real line
(Favard's theorem \cite{DPS}) and they are therefore called \emph{matrix
orthogonal polynomials}. The study of such polynomials goes back at least to
\cite{Krein} and we refer to the survey paper \cite{DPS} for a detailed
discussion of the available literature and for many more references.
Some recent developments and applications of matrix orthogonal polynomials can
be found in \cite{BCD,CGMV,DKS,GIM} among many others.

To the recurrence \eqref{recurrence} we associate the \emph{block Jacobi
matrix}
\begin{equation}\label{block:Jacobi}
J_n = \begin{pmatrix} B_0 & A_1 & & & 0 \\
A_1^* & B_1 & A_2 \\
& A_2^* & \ddots & \ddots \\
& & \ddots & \ddots & A_{n-1}\\
0 & & & A_{n-1}^* & B_{n-1}
\end{pmatrix}_{rn\times rn},
\end{equation}
which is a Hermitian, block tridiagonal matrix. It is well-known that
$$\det P_n(x) = c_n\det(xI_n-J_n)$$ with $c_n=\det(A_n^{-1}\cdots A_2^{-1}A_1^{-1})\neq 0$, see \cite{DPS,DL},
and with the \emph{zeros} of the matrix polynomial
$P_n(x)$ we mean the zeros of the determinant $\det
P_n(x)$, or equivalently the eigenvalues of the matrix $J_n$ defined
in \eqref{block:Jacobi} (counting multiplicities).

The polynomials $(P_n(x))_{n=0}^{\infty}$ are said to belong to the
\emph{matrix Nevai class} if the limits
\begin{equation}\label{Nevai:class}
\lim_{n\to\infty} A_n = A,\qquad \lim_{n\to\infty} B_n = B,
\end{equation}
exist, where we will assume throughout this paper that $A$ is nonsingular.

One of the famous classical results on orthogonal polynomials is Rakhmanov's
theorem \cite{Rak}, see \cite{AptLopRoc,DKS,Simon10} for a survey of the recent
advances in this direction. Rakhmanov's theorem for matrix orthogonal
polynomials on the real line is discussed in \cite{DKS,YM}. These results give
a sufficient condition on the spectral measure of the matrix orthogonal
polynomials, in order to have recurrence coefficients in the matrix Nevai
class, with limiting values $A=I_r$ and $B=0$. For a discussion of the matrix
Nevai class in this case, see \cite{Koz}.

Dur\'an \cite{Duran} shows that in the matrix Nevai class \eqref{Nevai:class},
the limiting matrix ratio
\begin{equation}\label{ratio}
R(x):=\lim_{n\to\infty} P_{n}(x)P_{n+1}^{-1}(x),\qquad x\in\cee\setminus
[-M,M],\end{equation} exists and depends analytically on $x\in\cee\setminus
[-M,M]$. Here $M>0$ is a constant such that all the zeros of all the matrix
polynomials $ P_n(x)$ are in $[-M,M]$. Moreover, it is also proved in
\cite{Duran} that $R(x)A^{-1}$ is the Stieltjes transform of the spectral
measure for the matrix Chebyshev polynomials of the second kind generated by
the constant recurrence coefficients $A$ and $B$; {see
Section~\ref{section:Cheb}}.

The present paper has several purposes. First we will give a different
formulation of Dur\'an's result on ratio asymptotics \cite{Duran}. In
particular we will give a self-contained proof of the existence of the limiting
ratio $R(x)$ and express it in terms of a quadratic eigenvalue problem. Second,
our approach can also be used to obtain ratio asymptotics for some extensions
to the matrix Nevai class. More precisely, we will allow the coefficients $A_n$
and $B_n$ to be slowly varying in function of a parameter, or to be
asymptotically periodic with higher periodicity. In both cases we prove the
existence of the limiting ratio (the limit may be {local or periodic})
and give an explicit formula for it.

To prove these results  we will work with \emph{normal families} in the sense
of Montel's theorem. Similar arguments can be found at various places in the
literature and our approach will be based in particular on the work by
Kuijlaars-Van Assche \cite{KVA} and its further developments in
\cite{BDK,CCV,Dette,Roman}. We point out that an alternative approach for
obtaining ratio asymptotics, is to use the \emph{generalized Poincar\'e
theorem} (see \cite{MN,Simon2} for this theorem); but the normal family
argument has the advantage that it also works for slowly varying recurrence
coefficients; see Section~\ref{subsection:N} for more details.
\newline\newline
The third purpose of the paper is to obtain
 a new  description of the limiting zero
distribution of the matrix polynomials $(P_n(x))_{n=0}^\infty$.
Dur\'an-L\'opez-Saff \cite{DLS} showed that in the matrix Nevai class
\eqref{Nevai:class}, and assuming that $A$ is Hermitian, then the zero
distribution of $P_n(x)$ has a limit for $n\to\infty$ in the sense of the weak
convergence of measures. They expressed the limiting zero distribution of
$P_n(x)$ in terms of the spectral measure for the matrix Chebyshev polynomials
of the first kind, see Section~\ref{section:DLS} for the details. In contrast
to the work of \cite{DLS} the results derived in the present paper are also
applicable in the case when the matrix $A$ is nonsingular but not necessarily
Hermitian.  Maybe not surprisingly, the limiting eigenvalue distribution of the
matrix $J_n$ in \eqref{block:Jacobi}--\eqref{Nevai:class} is the same as the
limiting eigenvalue distribution for $n\to\infty$ of the \emph{block Toeplitz
matrix}
\begin{equation}\label{block:Toeplitz}
T_n = \begin{pmatrix} B & A &&& 0\\
A^* & B & A \\
& A^* & \ddots & \ddots \\
& & \ddots & \ddots & A\\
0& & & A^* & B
\end{pmatrix}_{rn\times rn}.
\end{equation}
The {eigenvalue counting measure} of the matrix $T_n$ has a weak limit
for $n\to\infty$ \cite{Widom1} and a description of the limiting measure can be
obtained from \cite{BS2,Delvaux,Widom1}.

In this paper we establish the limiting zero distribution of the matrix
polynomials $P_n(x)$ as a consequence of our results on ratio asymptotics.
{We also find} the limiting zero distribution for the previously
mentioned extensions to the matrix Nevai class. {That is}, the
coefficients $A_n$ and $B_n$ are again allowed to be slowly varying in function
of a parameter, or to be asymptotically periodic with higher periodicity.

Incidentally, we mention that it is possible to devise an alternative, linear
algebra theoretic proof for the fact that the matrices $J_n$ and $T_n$ have the
same weak limiting eigenvalue distribution, using the fact that the block
Jacobi matrix \eqref{block:Jacobi} is Hermitian \cite{KSS}; however our
approach has the advantage that it can be also used in the non-Hermitian case,
at least in principle. This means that most of the  methodology derived in this
paper is also  applicable in the case when  the recurrence matrix
\eqref{block:Jacobi} generating the polynomials $P_n(x)$, is no longer
Hermitian, or when it has a larger band width in its block lower triangular
part, as in \cite{BDK}. In fact the key places where we use the Hermiticity are
in the proof of  Proposition~\ref{prop:Gamma0} and in the proof of
Lemma~\ref{lemma:DetteReuther}, see \cite{Dette}; but it is reasonable to
expect that both facts remain true for some specific non-Hermitian cases as
well. This may be an interesting topic for further research.
\\ \\
The remaining part of this paper is organized as follows. In the next section
we state our results for the case of the matrix Nevai class. In
Section~\ref{section:beyondNevai} we generalize these findings, to the context
of recurrence coefficients with slowly varying or asymptotically periodic
behavior. In Section~\ref{section:Cheb} we apply our results to find new
formulas for {the spectral measures of} the matrix Chebyshev polynomials
of the first and second kind. In Section~\ref{section:DLS} we relate our
formula for the limiting zero distribution to the one of Dur\'an-L\'opez-Saff
\cite{DLS}. In Section \ref{random} we indicate some potential applications in
the context of random matrices. In particular we derive the limiting eigenvalue
distribution for a class of random band matrices, which generalize the random
tridiagonal representations of the $\beta$-ensembles, {which were
introduced in} \cite{DumEdel}. Finally, Section~\ref{section:proofs} contains
the proofs of the main results.

\section{Statement of results: the matrix Nevai class}
\label{section:statement}

\subsection{The quadratic eigenvalue problem}

Throughout this section we work in the matrix Nevai class
\eqref{block:Jacobi}--\eqref{Nevai:class}. Observe that the limiting ratio
$R(x)$ in \eqref{ratio} satisfies the matrix relation
\begin{equation}\label{R:matrix:relation}
A^* R(x)+B-xI_r+AR^{-1}(x) = 0,\qquad x\in\cee\setminus [-M,M],
\end{equation}
which  is an easy consequence of the three term recurrence \eqref{recurrence}.
For a simple  motivation of our approach we
 assume at the moment that for each $x\in\cee\setminus[-M,M]$, the matrix
$R(x)$ is diagonalizable with distinct eigenvalues $z_k=z_k(x)$,
$k=1,\ldots,r$. Let $\vecv_k(x)\in\cee^r$ be the corresponding eigenvectors so
that
$$R(x)\vecv_k(x) = z_k(x)\vecv_k(x),$$ for $k=1,\ldots,r$. Multiplying \eqref{R:matrix:relation} on
the right with $\vecv_k(x)$ we find the relation
\begin{equation}\label{vector:v} \left(z_k(x) A^*+B-xI_r+z_k^{-1}(x)A\right)\vecv_k(x) = \mathbf{0},
\end{equation} where $\mathbf{0}$ denotes a column vector with all its entries equal to zero.
This relation implies in particular that
\begin{equation}\label{algebraic:equation:0} \det\left(z_k(x)A^*+B-xI_r+z_k^{-1}(x)A\right)=0.
\end{equation}
We can view \eqref{algebraic:equation:0} as a \emph{quadratic eigenvalue
problem} in the variable $z=z_k(x)$, and it gives us an algebraic equation for
the eigenvalues of the limiting ratio  $R(x)$. The corresponding eigenvectors $\vecv_k(x)$ can then
be found from \eqref{vector:v}. Note that the equation
\eqref{algebraic:equation:0} has $2r$ roots $z=z_k(x)$, $k=1,\ldots,2r$.
Ordering these roots by increasing modulus
\begin{equation}\label{rootszk:ordering}
0<|z_1(x)|\leq |z_2(x)|\leq\ldots\leq |z_r(x)|\leq |z_{r+1}(x)|\leq\ldots\leq
|z_{2r}(x)|,
\end{equation}
then we will see below that the eigenvalues of $R(x)$ are precisely the $r$
smallest modulus roots $z_1(x),\ldots,z_r(x)$.

In order to treat the general case of eigenvalues with multiplicity larger than
$1$ we now proceed in a more formal way. Inspired by the above discussion, we
define the algebraic equation
\begin{equation}\label{algebraic:equation}
0 = f(z,x) := \det(zA^*+B+z^{-1}A-x I_r),
\end{equation}
where $z$ and $x$ denote two complex variables. As mentioned, one may consider
\eqref{algebraic:equation} to be a (usual) eigenvalue problem in the variable
$x$ and a quadratic eigenvalue problem in the variable $z$.

Expanding the determinant in \eqref{algebraic:equation}, we can write it as a
Laurent polynomial in $z$:
\begin{equation}\label{algebraic:equation:Laurent} f(z,x)= \sum_{k=-r}^r f_k(x)z^k,
\end{equation}
where the coefficients $f_k(x)$ are polynomials in $x$ of degree at most $r$,
and with the outermost coefficients $f_r(x)$ and $f_{-r}(x)$ given by
$$ f_{r}(x)\equiv f_{r} = \det A^*,\qquad f_{-r}(x)\equiv f_{-r} = \det A.
$$


Solving the algebraic equation $f(z,x)=0$ for $z$ yields $2r$ roots (counting
multiplicities) $z_k = z_k(x)$ which we order by increasing modulus as in
\eqref{rootszk:ordering}.
If $x\in\cee$ is such that two or more roots $z_k(x)$ have the same modulus
then we may arbitrarily label them such that \eqref{rootszk:ordering} holds.

Define
\begin{equation}\label{Gamma0}
\Gamma_0 := \{x\in\cee\mid |z_r(x)|= |z_{r+1}(x)|\}.
\end{equation}
It turns out that the set $\Gamma_0$ attracts the eigenvalues of the block
Toeplitz matrix $T_n$ in \eqref{block:Toeplitz} for $n\to\infty$, see
\cite{BS2,Delvaux,Widom1}. This set will also attract the eigenvalues of the
matrix $J_n$ in \eqref{block:Jacobi}. The structure of $\Gamma_0$ is given in
the next proposition, which  is proved in
Section~\ref{subsection:proof:Gamma0}.

\begin{proposition} 
\label{prop:Gamma0} $\Gamma_0$ is a subset of the real line. It is compact and
can be written as the disjoint union of at most $r$ intervals.
\end{proposition}

\subsection{Ratio asymptotics}
\label{subsection:ratioasy}


For any $x\in\cee$ and any root $z=z_k(x)$ of the quadratic eigenvalue problem
\eqref{algebraic:equation},
we can choose a column vector $\vecv_k(x) \in \mathbb{C}^r$ in the right null
space of the matrix $z_k(x)A^*+B-xI_r+z_k^{-1}(x)A$, see \eqref{vector:v}.
If the null space is one-dimensional then we can uniquely fix the vector
$\vecv_k(x)$ by requiring it to have unit norm and a first nonzero component
which is positive.

If the null space is $2$- or more dimensional, then the vector $\vecv_k(x)$ is
not uniquely determined from \eqref{vector:v}. We need some terminology.
{(The following  paragraphs are a bit more technical and the reader may wish
to move directly to Theorem~\ref{theorem:ratioasy}.)} For any $x,z\in\cee$
define $d(z,x)$ as the geometric dimension of the null space in
\eqref{vector:v}, i.e.,
\begin{equation}\label{mult:geometric} d(z,x) := \dim\{ \vecv\in\cee^r\mid (A^*z+B+Az^{-1}-x
I_r)\vecv=\mathbf{0}\}.
\end{equation}
Also define the algebraic multiplicities
\begin{eqnarray} \label{mult:algebraic:z} m_{1}(z,x) &:=& \max\{ k\in\zet_{\geq 0}\mid (Z-z)^k\ \mathrm{ divides }\ Z^r f(Z,x)\},
\\ \label{mult:algebraic:x} m_2(z,x)&:=& \max\{ k\in\zet_{\geq 0}\mid (X-x)^k\ \mathrm{ divides }\ f(z,X)\},
\end{eqnarray}
where $Z$ and $X$ are auxiliary variables and the division is understood in the
ring of polynomials in $Z$ and $X$ respectively.

In the spirit of linear algebra, we can think of $d(z,x)$ as the
\emph{geometric multiplicity} of $(z,x)\in\cee^2$ while $m_1(z,x)$ and
$m_2(z,x)$ are the \emph{algebraic multiplicities} of $(z,x)$ with respect to
the variables $z$ and $x$ respectively.

The next lemma is a generalization of a result of Dur\'an
\cite[Lemma~2.2]{Duran}.

\begin{lemma}\label{lemma:alggeomult}
(Algebraic and geometric multiplicities:) Let $A,B\in\cee^{r\times r}$ be
matrices with $A$ non-singular and $B$ Hermitian. Define $f(z,x)$ by
\eqref{algebraic:equation}. Then the algebraic and geometric multiplicities in
\eqref{mult:geometric}--\eqref{mult:algebraic:x} are the same:
$$ d(z,x) = m_1(z,x) = m_2(z,x),
$$
{for all but finitely many $z,x\in\cee$}.
\end{lemma}

Lemma~\ref{lemma:alggeomult} is proved in
Section~\ref{subsection:proof:alggeo}. We note that the particular Hermitian
structure of the problem is needed in the proof.

\smallskip Now let $x\in\cee$ and consider the roots $z_1(x),\ldots,z_{2r}(x)$ in
\eqref{rootszk:ordering} with algebraic multiplicities taken into account.
Lemma~\ref{lemma:alggeomult} ensures that for all but finitely many $x\in\cee$,
we can find corresponding vectors $\vecv_1(x),\ldots,\vecv_{2r}(x)$ having unit
norm, satisfying \eqref{vector:v}, and such that
\begin{equation}\label{diagonalizable} \dim\{ \vecv_l(x)\mid l\in\{1,\ldots,2r\}\textrm{ with
}z_l(x)=z_k(x)\}=d(z_k(x),x)=m_1(z_k(x),x),\end{equation} for each fixed $k$.
In what follows we will always assume that the vectors $\vecv_k(x)$ are chosen
in this way. We write $S\subset\cee$ for the set of those $x\in\cee$ for which
\eqref{diagonalizable} \emph{cannot} be achieved. Thus the set $S$ has a finite
cardinality.

\smallskip Now we are ready to describe the ratio asymptotics for the matrix
Nevai class. The next result should be compared to the one of Dur\' an
\cite{Duran}.

\begin{theorem}\label{theorem:ratioasy} (Ratio asymptotics.)
Let $A,B\in\cee^{r\times r}$ be matrices with $A$ non-singular and $B$
Hermitian. Let $P_n(x)$ satisfy \eqref{recurrence} and \eqref{Nevai:class}. Let
$M>0$ be such that all the zeros of all the polynomials $\det P_n(x)$ are in
$[-M,M]$, and let
$S\subset \cee$ be the set of finite cardinality defined in the previous
paragraphs. Then for all $x\in\cee\setminus ([-M,M]\cup S)$ the limiting
$r\times r$ matrix
$$ \lim_{n\to\infty} P_{n}(x)P_{n+1}^{-1}(x)$$ exists entrywise and is diagonalizable,
with
$$ \left(\lim_{n\to\infty} P_{n}(x)P_{n+1}^{-1}(x)\right)\vecv_k(x) =
z_k(x)\vecv_k(x),\qquad k=1,\ldots,r,
$$
uniformly for $k\in\{1,\ldots,r\}$ and for $x$ on compact subsets of
$\cee\setminus ([-M,M]\cup S)$. Here we take into account multiplicities as
explained in the paragraphs before the statement of the theorem.
\end{theorem}


Theorem~\ref{theorem:ratioasy} is proved in
Section~\ref{subsection:proof:ratio}.
See also Section~\ref{section:beyondNevai}
for the generalization of Theorem~\ref{theorem:ratioasy} beyond the matrix
Nevai class.

Since the determinant of a matrix is the product of its eigenvalues,
Theorem~\ref{theorem:ratioasy} implies:

\begin{corollary}\label{cor:ratiodet} Under the assumptions of Theorem~\ref{theorem:ratioasy} we have
$$ \lim_{n\to\infty}\frac{\det P_{n}(x)}{\det P_{n+1}(x)} = z_1(x)\ldots z_r(x), $$
uniformly on compact subsets of $\cee\setminus [-M,M]$.
\end{corollary}

Note that the convergence in the previous result holds in $\cee\setminus
[-M,M]$ rather than $\cee\setminus ([-M,M]\cup S)$. This is due to
Lemma~\ref{lemma:DetteReuther}. See Section~\ref{section:Cheb} below for some
further corollaries of Theorem~\ref{theorem:ratioasy} in terms of the matrix
Chebychev polynomials.

\subsection{Limiting zero distribution}

With the above results on ratio asymptotics in place, it is a rather standard
routine to obtain the limiting zero distribution for the polynomials
$(P_n(x))_{n=0}^\infty$ in the matrix Nevai class. {Recall that the
zeros of the matrix polynomial $P_n(x)$ are defined as the zeros of $\det
P_n(x)$,
or equivalently the eigenvalues of the Hermitian matrix $J_n$ in
\eqref{block:Jacobi}}. If these zeros are denoted by $x_1\leq x_2\leq
\ldots\leq x_{rn}$ (taking into account multiplicities), then we define the
\emph{normalized zero counting measure} by
\begin{equation}\label{count}
 \nu_n = \frac{1}{rn}\sum_{k=1}^{rn}\delta_{x_k},
\end{equation}
where $\delta_x$ is the Dirac measure at the point $x$.

\begin{theorem}\label{theorem:zerodistr}
Under the assumptions of Theorem~\ref{theorem:ratioasy} the normalized zero
counting measure $\nu_n$ defined in \eqref{count} has a weak limit $\mu_0$ for $n\to\infty$. The
(probability) measure $\mu_0$ is supported on the set $\Gamma_0$ defined in \eqref{Gamma0} and has logarithmic
potential
\begin{equation}\label{limiting:measure:mu0}
\int \log|x-t|^{-1}\ \ud\mu_0(t) = \frac 1r \log|z_1(x)\ldots z_r(x)|+C,\qquad
x\in\cee\setminus\Gamma_0,
\end{equation}
{for some explicit constant $C$ (actually $C=-\frac 1r \log|\det A|$.)}
\end{theorem}

Recall that a measure $\mu_0$ on the real line is completely determined from
its logarithmic potential \cite{SaffTotik}.

Theorem~\ref{theorem:zerodistr} will be proved in
Section~\ref{subsection:proof:zero}, with the help of
Corollary~\ref{cor:ratiodet}. In the proof we will obtain a stronger version of
\eqref{limiting:measure:mu0}, with the absolute value signs in the logarithms
removed. Moreover, in Section~\ref{section:beyondNevai} we will extend
Theorem~\ref{theorem:zerodistr} beyond the matrix Nevai class.


It can be shown that the measure $\mu_0$ in Theorem~\ref{theorem:zerodistr} is
absolutely continuous on $\Gamma_0\subset\er$ with density (see also
\cite{BDK,Delvaux})
\begin{equation} \label{measure:mu0} \ud\mu_{0}(x) = \frac{1}{r}\frac{1}{2\pi i}\sum_{j=1}^{r}\left(
\frac{z_{j+}'(x)}{z_{j+}(x)}-\frac{z_{j-}'(x)}{z_{j-}(x)} \right)\ud x,\qquad
x\in\Gamma_0.
\end{equation}
Here the prime denotes the derivation with respect to $x$, and $z_{j+}(x)$ and
$z_{j-}(x)$ are the boundary values of $z_j(x)$ obtained from the upper and
lower part of the complex plane respectively. These boundary values exist for
all but a finite number of points $x\in\Gamma_0$. See Section~\ref{section:DLS}
for a comparison to the formulas of \cite{DLS}.


\subsection{Alternative description of the limiting zero distribution}
\label{subsection:mu0alternative}

First we give an alternative description of the set $\Gamma_0$ in
\eqref{Gamma0}. For the proof see Section~\ref{subsection:proof:Gamma0bis}.

\begin{proposition}\label{prop:Gamma0:bis} We have \begin{equation}\label{Gamma0:unitnorm}
\Gamma_0 = \{x\in\cee\mid \exists \textrm{$z\in\cee$ with $f(z,x)=0$ and
$|z|=1$} \}\subset\er.\end{equation} 
In words, $\Gamma_0$ is the set of all points $x$ for which $f(z,x)=0$ has a
root with unit modulus.
\end{proposition}


Now let $$\mathcal I:=(x_1,x_2)\subset\Gamma_0$$ be an open interval disjoint
from the set of branch points of the algebraic equation $f(z,x)=0$. {We can
then choose a labeling of the roots so that each $z_k(x)$,
$x\in \mathcal I$, is the restriction to $\mathcal I$ of an analytic function
defined in an open complex neighborhood $\Omega\supset \mathcal I$. Note that
we do not insist to have the ordering \eqref{rootszk:ordering} anymore}. If $k$
is such that $|z_k(x)|=1$ throughout the interval $\mathcal I$ then we write
$$z_k(x)=e^{i\theta_k(x)},\qquad x\in \mathcal I,$$ with $\theta_k$ a real valued, differentiable argument function on
$\mathcal I$. Observe that
$$ \frac{z_k'(x)}{z_k(x)} = i\theta_k'(x),\qquad x\in \mathcal I.
$$
Moreover $\theta_k'(x)$ describes how fast $z_k(x)$ runs on the unit circle in
function of $x\in \mathcal I$.

We can now give an alternative description of the measure $\mu_0$ in
\eqref{measure:mu0} in the following  proposition which
is proved in
Section~\ref{subsection:proof:mu0bis}.

\begin{proposition}\label{prop:mu0:bis} With the above notations we have
\begin{eqnarray} \label{measure:mu0:bis} \frac{\ud\mu_{0}(x)}{\ud x} &=& \frac{1}{2\pi r}
\sum_{k: |z_k(x)|=1}\left| \frac{z_{k}'(x)}{z_{k}(x)} \right|\\
&=& \frac{1}{2\pi r} \sum_{k: |z_k(x)|=1}\left| \theta_k'(x)\right|,\qquad x\in
\mathcal I.
\end{eqnarray}
Moreover, $\theta_k'(x)\neq 0$ for any $x\in \mathcal I$ and for any $k$ with
$|z_k(x)|=1$.
\end{proposition}

\section{Generalizations of the matrix Nevai class}
\label{section:beyondNevai}

\subsection{Slowly varying recurrence coefficients}
\label{subsection:N}

In this section we consider a first type of generalization of the matrix Nevai
class. We will assume that the recurrence coefficients $A_n$ and $B_n$ depend
on an additional variable $N>0$, as in Dette-Reuther \cite{Dette}. We write
$A_{n,N}$ and $B_{n,N}$. For each fixed $N>0$ we define the matrix-valued
polynomials $P_{n,N}(x)$ generated by the recurrence (compare with
\eqref{recurrence})
\begin{equation}\label{recurrence:N} xP_{n,N}(x) = A_{n+1,N} P_{n+1,N}(x)+B_{n,N}
P_{n,N}(x)+A_{n,N}^* P_{n-1,N}(x),\qquad n\geq 0,
\end{equation}
with again the initial conditions $P_{0,N}=I_r$ and $P_{-1,N}=0$.

Assume that the following limits exist:
\begin{equation}\label{Nevai:class:N}
\lim_{n/N\to s} A_{n,N}=:A_{s},\qquad \lim_{n/N\to s} B_{n,N}=:B_{s},
\end{equation}
for each $s>0$, where $\lim_{n/N\to s}$ means that we let both $n,N$ tend to
infinity in such a way that the ratio $n/N$ converges to $s>0$. We assume that each
$A_{s}$ is non-singular. We trust that the notation $A_s, B_s$ ($s>0$) will not
lead to confusion with our previous usage of $A_n,B_n$ ($n\in\zet_{\geq 0}$).

For each $s>0$ define the algebraic equation
\begin{equation}\label{algebraic:equation:N}
0 = f_{s}(z,x) := \det(A_{s}^*z+B_{s}+A_{s} z^{-1}-x I_r).
\end{equation}
We again define the roots $z_k(x,s)$, $k=1,\ldots,2r$ to this equation, ordered
as in \eqref{rootszk:ordering}, and we find the corresponding null space
vectors $\vecv_k(x,s)$ as in \eqref{vector:v}.  We also define the finite
cardinality set $S_s\subset\cee$ as before, and we define $\Gamma_0(s)$ as in
\eqref{Gamma0}.

We now formulate the extensions of Theorems~\ref{theorem:ratioasy} and
\ref{theorem:zerodistr} to the present setting.

\begin{theorem}\label{theorem:ratioasy:N}
Fix $s>0$ and assume \eqref{recurrence:N} and \eqref{Nevai:class:N} with the
limits $A_s,B_s$ depending continuously on $s\geq 0$. Then for all
$x\in\cee\setminus ([-M,M]\cup S_s)$ the limiting $r\times r$ matrix
$$ \lim_{n/N\to s} P_{n,N}(x)P_{n+1,N}^{-1}(x)$$ exists entrywise and is diagonalizable,
with
$$ \left(\lim_{n/N\to s} P_{n,N}(x)P_{n+1,N}^{-1}(x)\right)\vecv_k(x,s) =
z_k(x,s)\vecv_k(x,s),\qquad k=1,\ldots,r,
$$
uniformly for $x$ on compact subsets of $\cee\setminus ([-M,M]\cup S_s)$. We
also have $$ \lim_{n/N\to s}\frac{\det P_{n,N}(x)}{\det P_{n+1,N}(x)} =
z_1(x,s)\ldots z_r(x,s),
$$ uniformly on compact subsets of $\cee\setminus [-M,M]$.

\end{theorem}

\begin{theorem}\label{theorem:zerodistr:N}
Under the same assumptions as in Theorem~\ref{theorem:ratioasy:N}, the
normalized zero counting measure of $\det P_{n,N}(x)$ for $n/N\to s$ has a weak
limit $\mu_{0,s}$, with logarithmic potential given by
\begin{equation}\label{limiting:measure:mu0:N}
\int \log|x-t|^{-1}\ \ud\mu_{0,s}(t) = \frac{1}{rs}\int_{0}^s
\log|z_1(x,u)\ldots z_r(x,u)|\, \ud u+C_s,\qquad
x\in\cee\setminus\bigcup_{0\leq u\leq s}\Gamma_0(u),
\end{equation}
for some explicit constant $C_s$.
\end{theorem}

In other words, $\mu_{0,s}$ is precisely the average (or integral) of the
individual limiting measures for fixed $u$, integrated over $u\in [0,s]$.

Theorems~\ref{theorem:ratioasy:N} and \ref{theorem:zerodistr:N} are proved in
Section~\ref{subsection:proof:N} and an application of these results in the
context of random band matrices will be given in Section \ref{random}.

\subsection{Asymptotically periodic recurrence coefficients}
\label{subsection:p}

In this section we consider a second type of generalization of the Nevai class.
We will assume that the matrices $A_n$ and $B_n$ have \emph{periodic} limits
with period $p\in\zet_{>0}$ in the sense that
\begin{equation}\label{Nevai:class:period}
\lim_{n\to\infty} A_{pn+j} =: A^{(j)},\qquad \lim_{n\to\infty} B_{pn+j} =:
B^{(j)},
\end{equation}
for any fixed $j=0,1,\ldots,p-1$.

The role that was previously played by $zA^* +B+z^{-1}A-xI_r$ will now be
played by the matrix
\begin{equation}\label{symbol:p}
F(z,x):=\begin{pmatrix} B^{(0)}-xI_r & A^{(1)} & 0 & 0 & z A^{(0)*} \\
A^{(1)*} & B^{(1)}-xI_r & A^{(2)} & 0 & 0 \\
0 & A^{(2)*} & \ddots & \ddots & 0 \\
0 & 0 & \ddots & \ddots & A^{(p-1)}\\
z^{-1} A^{(0)}  & 0 & 0 & A^{(p-1)*} & B^{(p-1)}-xI_r
\end{pmatrix}_{pr\times pr},
\end{equation}
{where we abbreviate $A^{(j)*}:=(A^{(j)})^*$}. We define
\begin{equation}\label{algebraic:equation:p} f(z,x):=\det F(z,x).\end{equation}
This can be expanded as a Laurent polynomial in $z$:
\begin{equation}\label{algebraic:equation:Laurent:bis} f(z,x)= \sum_{k=-r}^r f_k(x)z^k,
\end{equation}
where now the outermost coefficients $f_r(x)$, $f_{-r}(x)$ are given by
$$ f_{r}(x)\equiv f_{r} = (-1)^{(p-1)r^2}\det(A^{(0)*}\ldots A^{(p-1)*}),\qquad f_{-r}(x)\equiv f_{-r} = (-1)^{(p-1)r^2}\det(A^{(0)}\ldots A^{(p-1)}).
$$
We again order the roots $z_k(x)$, $k=1,\ldots,2r$ of
\eqref{algebraic:equation:Laurent:bis} as in \eqref{rootszk:ordering}. We
define $\Gamma_0$ as in \eqref{Gamma0}. Proposition~\ref{prop:Gamma0} now takes
the following form:

\begin{proposition} 
\label{prop:Gamma0:p} $\Gamma_0$ is a subset of the real line. It is compact
and can be written as the disjoint union of at most $pr$ intervals.
\end{proposition}

As before we denote with $\vecv_k(x)$, $k=1,\ldots,2r$, the normalized null
space vectors of the matrix $F(z_k(x),x)$. We partition these vectors in blocks
as
\begin{equation}\label{nullspacevector} \vecv_{k}(x) =
\begin{pmatrix}
\vecv_{k,0}(x)\\ \vdots \\ \vecv_{k,p-1}(x)
\end{pmatrix},
\end{equation}
where each $\vecv_{k,j}(x)$, $j=0,1,\ldots,p-1$, is a column vector of length
$r$. Lemma~\ref{lemma:alggeomult} remains true in the present setting.


Now we can state the analogues of Theorems~\ref{theorem:ratioasy} and
\ref{theorem:zerodistr}.

\begin{theorem}\label{theorem:ratioasy:p}
For all $x\in\cee\setminus ([-M,M]\cup S)$ and any $j\in\{0,1,\ldots,p-1\}$,
the limiting $r\times r$ matrix
$$ \lim_{n\to\infty} P_{pn+j}(x)P_{pn+j+p}^{-1}(x)$$ exists entrywise and is diagonalizable,
with
\begin{equation}\label{ratio:asy:p:weak}
\left(\lim_{n\to\infty} P_{pn+j}(x)P_{pn+j+p}^{-1}(x)\right)\vecv_{k,j}(x) =
z_k(x)\vecv_{k,j}(x),\end{equation} uniformly for $k\in\{1,\ldots,r\}$ and for
$x$ in compact subsets of $\cee\setminus ([-M,M]\cup S)$. Moreover,
\begin{equation} \lim_{n\to\infty} \frac{\det P_{n}(x)}{\det P_{n+p}(x)} =
z_1(x)\ldots z_r(x),\end{equation} uniformly for $x$ in compact subsets of
$\cee\setminus [-M,M]$.
\end{theorem}

Theorem~\ref{theorem:ratioasy:p} is proved in Section~\ref{subsection:proof:p},
where we will establish in fact a stronger variant \eqref{ratio:asy:p:strong}
of \eqref{ratio:asy:p:weak}. The proof will use some ideas from \cite{BDK}.
Note that the eigenvectors in \eqref{ratio:asy:p:weak} depend on the residue
class modulo $p$, $j\in\{0,1,\ldots,p-1\}$, while the eigenvalues are
independent of $j$.

\begin{theorem}\label{theorem:zerodistr:p}
Under the same assumptions as in Theorem~\ref{theorem:ratioasy:p}, the
normalized zero counting measure of $\det P_{n}(x)$ for $n\to \infty$ has a
weak limit $\mu_0$, supported on $\Gamma_0$, with logarithmic potential given
by
\begin{equation}\label{limiting:measure:mu0:p}
\int \log|x-t|^{-1}\ \ud\mu_0(t) = \frac{1}{pr}\log|z_1(x)\ldots
z_r(x)|+C,\qquad x\in\cee\setminus\Gamma_0,
\end{equation}
for some explicit constant $C$ (actually
$C=-\frac{1}{pr}\log|\det(A^{(0)}\ldots A^{(p-1)})|$.)
\end{theorem}

Finally we point out that the results in the present section can be combined
with those in Section~\ref{subsection:N}. That is, one could allow for slowly
varying recurrence coefficients $A_{n,N}$, $B_{n,N}$ with periodic local limits
of period $p$. The corresponding modifications are straightforward and are left
to the interested reader.

\section{Formulas for matrix Chebychev polynomials}
\label{section:Cheb}

Throughout this section we let $A$ be a fixed nonsingular, and $B$ a fixed
Hermitian $r\times r$ matrix. {The matrix Chebyshev polynomials of the
second kind are defined from the recursion \eqref{recurrence}, with the
standard initial conditions $P_0(x)\equiv I_r$, $P_{-1}(x)\equiv 0$, and with
constant recurrence coefficients $A_n\equiv A$ and $B_n\equiv B$ for all $n$.
The matrix Chebyshev polynomials of the first kind are defined in the same way,
but now with $A_1=\sqrt{2}A$ and $A_n=A$ for all $n\geq 2$.}

Let $X$ and $W$ be the spectral measures for the matrix Chebychev polynomials
of the first and second kind respectively, normalized in such a way that
$$ \int_{\er} \ud X = \int_{\er} \ud W = I_r,
$$
as in \cite{Duran,DLS}. Here the integrals are taken entrywise. Denote the
corresponding Stieltjes transforms (or Cauchy transforms) by
$$ F_X(x)=\int \frac{\ud X(t)}{x-t},\qquad F_W(x)=\int \frac{\ud W(t)}{x-t}.
$$

Theorem~\ref{theorem:ratioasy} can be reformulated as
\begin{equation}\label{ratio:asy:VD} \lim_{n\to\infty} P_{n}(x)P_{n+1}^{-1}(x)
= V(x)D(x)V^{-1}(x),\qquad x\in\cee\setminus ([-M,M]\cup S),
\end{equation}
where $D(x)$ is the diagonal matrix with entries $z_k(x)$, $k=1,\ldots,r$ and
$V(x)$ is the matrix whose columns are the corresponding vectors $\vecv_k(x)$
in \eqref{vector:v}.

In turns out that the Stieltjes transforms $F_W$ and $F_X$ can be expressed
{in terms of the matrices $D=D(x)$ and $V=V(x)$ as well}.

\begin{proposition}\label{prop:Cheb12}
We have
\begin{equation}\label{Cheb2:cauch1} F_W(x) = V(x)D(x)V^{-1}(x)A^{-1}\end{equation}
and \begin{eqnarray} \label{Cheb1:cauch1}  F_X(x) &=& \left[ xI_r-B-2A^* F_W(x)A \right]^{-1} \\
\label{Cheb1:cauch2}  &=& V\left[AVD^{-1}-A^*VD\right]^{-1},
\end{eqnarray}
{for all but finitely many $x\in\cee\setminus \Gamma_0$. 
Hence the matrix-valued measures $W$ and $X$ are both supported on $\Gamma_0$
together with a finite, possibly empty set of mass points on $\er$.}
\end{proposition}

\begin{proof}
By comparing \eqref{ratio:asy:VD} with Dur\'an's result \cite[Thm.~1.1]{Duran}
we get the claimed expression \eqref{Cheb2:cauch1}. The formula
\eqref{Cheb1:cauch1} follows from the theory of the matrix continued fraction
expansion, see \cite{AN} and also \cite{Zyg}. Formula \eqref{Cheb1:cauch2} is
then a consequence of \eqref{Cheb2:cauch1}--\eqref{Cheb1:cauch1} and the matrix
relation \begin{equation}\label{matrixeq:VD} -(xI_r-B)V+A^* VD+AVD^{-1} = 0,
\end{equation}
which is obvious from the definitions of $D=D(x)$ and $V=V(x)$.

{To prove the remaining claims, we note that the right hand side of
\eqref{Cheb2:cauch1} is analytic for $x\in\cee\setminus (\Gamma_0\cup S\cup
\til S)$ with $\til S$ the set of branch points of \eqref{algebraic:equation},
so $F_W(x)$ is also analytic there and therefore the measure $W$ has its
support in a subset of $(\Gamma_0\cup S\cup \til S)\cap\er$. Finally, the
determinant of the matrix in square brackets in \eqref{Cheb1:cauch1} is
analytic and not identically zero for $x\in\cee\setminus (\Gamma_0\cup S\cup
\til S)$ so it has only finitely many zeros there. This yields the claim about
the support of the measure $X$.}
\end{proof}

The above descriptions considerably simplify if $A$ is Hermitian. In that case
the algebraic equation \eqref{algebraic:equation} becomes
\begin{equation}\label{algebraic:equation:w}
0=f(z,x) = \det(Aw+B-x I_r),
\end{equation}
where \begin{equation}\label{def:w} w:=z+z^{-1}.
\end{equation}

{As in Section~\ref{subsection:mu0alternative}, for any  open interval
$\mathcal I:=(x_1,x_2)\subset\Gamma_0$ disjoint from the set of branch points
of the algebraic equation \eqref{algebraic:equation}, we can choose an ordering
of the roots $z_k(x)$, $x\in \mathcal I$, so that each $z_k(x)$ is the
restriction to $\mathcal I$ of an analytic function defined on an open complex
neighborhood $\Omega\supset \mathcal I$. Thereby we drop the ordering
constraint \eqref{rootszk:ordering}. By
\eqref{algebraic:equation:w}--\eqref{def:w} we may assume that
\begin{equation}\label{assumption:Herm1}z_{2r-k}(x)=z_k(x)^{-1},\qquad
k=1,\ldots,r,\end{equation} for all $x\in \Omega\supset \mathcal I$ and we
write}
\begin{equation}\label{assumption:Herm2} w_k(x)=z_k(x)+z_k(x)^{-1},\qquad
k=1,\ldots,r.
\end{equation}

If $w_k(x)\in (-2,2)$ for all $x\in \mathcal I$ then write
\begin{equation}\label{assumption:Herm3}w_k(x)=2\cos \theta_k(x),\qquad
0<\theta_k(x)<\pi,\end{equation} with $\theta_k$ a real-valued, differentiable
argument function on $\mathcal I$. Denote with $V(x)$ the matrix formed by the
normalized null space vectors $\vecv_k(x)$ for the roots $w_k(x)$, $x\in
\mathcal I$.

\begin{proposition}\label{prop:Cheb12bis} Assume that $A$ is Hermitian. Then with the above
notations,
the density of the absolutely continuous part of the measures $X$ and $W$ is
given by
$$ \frac{\ud X(x)}{\ud x} = V(x)\Lambda_X(x)V^{-1}(x)A^{-1},\qquad x\in\Gamma_0,
$$
$$ \frac{\ud W(x)}{\ud x} = V(x)\Lambda_W(x)V^{-1}(x)A^{-1},\qquad x\in\Gamma_0,
$$
where
$$ \Lambda_X(x) = \frac{1}{\pi}\diag\left(\frac{\mathbf 1_{w_k(x)\in(-2,2)}}{\sqrt{4-w_k(x)^2}}\ \mathrm{sign}\,{w_k'(x)}\right)_{k=1}^r,
$$
$$ \Lambda_W(x) = \frac{1}{2\pi}\diag\left(\mathbf 1_{w_k(x)\in(-2,2)}\sqrt{4-w_k(x)^2}\
\mathrm{sign}\,{w_k'(x)}\right)_{k=1}^r,
$$
and where the characteristic function $\mathbf 1_{w_k\in (-2,2)}$ takes the
value $1$ if $w_k\in (-2,2)$ and zero otherwise.
\end{proposition}

Proposition~\ref{prop:Cheb12bis} is proved in
Section~\ref{subsection:proof:cheb}.  If $A$ is positive definite then the
factors $\mathrm{sign}\,{w_k'(x)}$ in the above formulas can be removed. This
follows from \eqref{connection:with:DLS} below. Then one can show that the
above formulas correspond to those in \cite{Duran,DLS}.

\section{The results of Dur\' an-L\'opez-Saff revisited.} \label{section:DLS}

In this section we show how Theorem~\ref{theorem:zerodistr} on the limiting
zero distribution of $P_n(x)$ in the matrix Nevai class, relates to the
formulas of Dur\' an-L\'opez-Saff~\cite{DLS} for the case where the matrix $A$
is positive definite or Hermitian.

Throughout this section we write the algebraic equation $f(z,x)=0$ as in
\eqref{algebraic:equation:w}--\eqref{def:w}. First we will assume that $A$ is
positive definite. Then $A^{1/2}$ exists and we can replace the algebraic
equation \eqref{algebraic:equation:w} by
\begin{equation*}
0= \det(w I_r+A^{-1/2}BA^{-1/2}-x A^{-1}).
\end{equation*}
Hence the roots $w$ are the eigenvalues of the matrix $$ x
A^{-1}-A^{-1/2}BA^{-1/2}.$$ If $x\in\er$ then this matrix is Hermitian and we
denote its spectral decomposition by \begin{equation}\label{spectral:decomp} x
A^{-1} - A^{-1/2}BA^{-1/2} = U(x)D_w(x)U^{-1}(x),
\end{equation}
where $$D_w(x)=\diag(w_1(x),\ldots,w_r(x))$$ is the diagonal matrix containing
the eigenvalues, and $U(x)$ is the corresponding eigenvector matrix. We can
assume that $U(x)$ is unitary, i.e.~$U^{-1}(x) = U^*(x)$.

{As in the previous section, we fix an open interval $\mathcal
I:=(x_1,x_2)\subset\Gamma_0$ disjoint from the set of branch points of the
algebraic equation \eqref{algebraic:equation}, and we choose an ordering of the
roots $z_k(x)$, $x\in \mathcal I$, so that each $z_k(x)$ is the restriction to
$\mathcal I$ of an analytic function defined on an open complex neighborhood
$\Omega\supset \mathcal I$. The same then holds for $w_k(x)$ in
\eqref{assumption:Herm2}}.


\begin{lemma} If $A$ is positive definite, then with the above notations we have
\begin{equation}\label{connection:with:DLS} w_k'(x) =
(U^{-1}(x)A^{-1}U(x))_{k,k}>0,\qquad k=1,\ldots,p,
\end{equation}
where we use the notation $M_{k,k}$ to denote the $(k,k)$ entry of a matrix
$M$.
\end{lemma}

\begin{proof} We take the derivative of \eqref{spectral:decomp} with respect to
$x$. This yields (we suppress the dependence on $x$ for simplicity)
$$ A^{-1}= U'D_w U^{-1}+UD_w'U^{-1}+UD_w(U^{-1})',
$$
or
$$ U^{-1}A^{-1}U= U^{-1}U'D_w+D_w'+D_w(U^{-1})'U.
$$
Invoking the fact that $(U^{-1})'=-U^{-1}U' U^{-1}$, this becomes
$$ U^{-1}A^{-1}U= D_w'+[U^{-1}U',D_w]~,
$$
where the square brackets denote the commutator. The equality in
\eqref{connection:with:DLS} then follows on taking the $(k,k)$ diagonal entry
of this matrix relation, and noting that the diagonal entries of the commutator
$[U^{-1}U',D_w]$ are all zero since $D_w$ is diagonal. Finally, the inequality
in \eqref{connection:with:DLS} follows because $A$ is positive definite and $U$
is unitary, $U^{-1}=U^*$.
\end{proof}

From \eqref{assumption:Herm2} we have \begin{equation}\label{pm:sign}
\frac{z_k'(x)}{z_k(x)}=\pm i\frac{w_k'(x)}{\sqrt{4-w_k(x)^2}}.\end{equation}

Thus the density of the limiting zero distribution $\mu_0$ in
\eqref{measure:mu0:bis} becomes \begin{equation}\label{measure:mu0:interm}
\frac{\ud\mu_0(x)}{\ud x} = \frac{1}{\pi r}\sum_{k=1}^r
\frac{w_k'(x)}{\sqrt{4-w_k(x)^2}}\mathbf 1_{w_k\in(-2,2)},
\end{equation}
where we used that $w_k'(x)>0$, and where again the characteristic function
$\mathbf 1_{w_k\in(-2,2)}$ takes the value $1$ if $w_k\in(-2,2)$ and zero
otherwise. Note that the factor $2$ in the denominator of
\eqref{measure:mu0:bis} is canceled since for any $w_k\in (-2,2)$ there are
\emph{two} solutions $z_k$ to \eqref{assumption:Herm2}, one leading to the plus
and the other to the minus sign in \eqref{pm:sign}.

Inserting \eqref{connection:with:DLS} in \eqref{measure:mu0:interm} we get
\begin{eqnarray*} \frac{\ud\mu_0(x)}{\ud x} &=& \frac{1}{\pi r}\sum_{k=1}^r
\frac{(U^{-1}(x)A^{-1}U(x))_{k,k}}{\sqrt{4-w_k(x)^2}}\mathbf
1_{w_k\in(-2,2)}\\
&=& \frac{1}{r}\tr\left( U^{-1}(x)A^{-1}U(x)\Lam_X(x)\right),
\end{eqnarray*}
where $\tr (C)  $ denotes the trace of the matrix $C$ and $\Lam_X(x)$ is the
diagonal matrix $$ \Lam_X(x) = \frac{1}{\pi}\diag\left(\frac{\mathbf
1_{w_k\in(-2,2)}}{\sqrt{4-w_k(x)^2}}\right)_{k=1}^r.
$$
Since the trace of a matrix product is invariant under cyclic permutations, we
find
$$
\frac{\ud\mu_0(x)}{\ud x} = \frac{1}{r}\tr\left(
A^{-1/2}U(x)\Lam_X(x)U^{-1}(x)A^{-1/2}\right).
$$
This is the result of Dur\' an-L\'opez-Saff for the case where $A$ is positive
definite \cite{DLS}.

\smallskip
Next suppose that $A$ is Hermitian but not necessarily positive definite. Then
we can basically repeat the above procedure. Rather than taking the square root
$A^{1/2}$, we now write \eqref{algebraic:equation:w} in the form
\begin{equation*}
0= \det(w I_r+A^{-1}B-x A^{-1}).
\end{equation*}
Hence the roots $w$ are the eigenvalues of the matrix $x A^{-1}-A^{-1}B$.
Supposing that this matrix is diagonalizable, we write
\begin{equation}\label{spectral:decomp:bis} x A^{-1} - A^{-1}B =
V(x)D_w(x)V^{-1}(x),
\end{equation}
with $D_w=\diag(w_1,\ldots,w_r)$ but with $V$ not necessarily unitary. The
notation $V$ is compatible with the one in the previous section, by virtue of
\eqref{matrixeq:VD}. We then again have the equality in
\eqref{connection:with:DLS} (with $U$ replaced by $V$) although the positivity
$w_k'(x)>0$ may be violated. Similarly to the above paragraphs, we find
$$
\frac{\ud\mu_0(x)}{\ud x} = \frac{1}{r}\tr\left(
V(x)\Lam_X(x)V^{-1}(x)A^{-1}\right),
$$
with $\Lam_X$ defined in Proposition~\ref{prop:Cheb12bis}. The latter
proposition then implies that
$$
\frac{\ud\mu_0(x)}{\ud x} = \frac{1}{r}\tr \left(\frac{\ud X(x)}{\ud x}\right).
$$
This is the result of Dur\' an-L\'opez-Saff for the case where $A$ is Hermitian
\cite{DLS}.

\section{An application to random matrices}
\label{random}

 In a fundamental paper
Dumitriu-Edelman \cite{DumEdel} introduced a tridiagonal random matrix
\begin{eqnarray*} \label{g1}
G^{(1)}_n=
 \left(
   \begin{array}{cccccc}
     N_1    & \frac {1}{\sqrt{2}} \mathcal{X}_{(n-1)\beta}  &   &  &  & \\
    \frac {1}{\sqrt{2}} \mathcal{X}_{(n-1)\beta}   & N_2  & \frac {1}{\sqrt{2}} \mathcal{X}_{(n-2)\beta}  &  &  &  \\
       & \frac {1}{\sqrt{2}} \mathcal{X}_{(n-2)\beta}  & N_3  &  &  &  \\
       &   &   & \ddots &  &  \\
       &   &   &  &  N_{n-1} & \frac {1}{\sqrt{2}} \mathcal{X}_{\beta}  \\
       &   &   &  & \frac {1}{\sqrt{2}} \mathcal{X}_{\beta} & N_n  \\
   \end{array}
 \right)
\end{eqnarray*}
where $N_1,\dots,N_n$ are independent identically  standard normal distributed
random variables and $\mathcal{X}^2_{1 \beta}, \dots,
\mathcal{X}^2_{(n-1)\beta}$ are independent random variables also independent
of $N_1,\dots,$ $N_n$, such that $\mathcal{X}^2_{j \beta}$ is a chi-square
distribution with $j \beta$ degrees of freedom. They showed that the density of
the eigenvalues $\lambda_1 \leq \dots \leq \lambda_n$ of the matrix $G_n$ is
given by the so called  beta ensemble
$$
c_\beta \prod_{i < j} | \lambda_i - \lambda_j |^\beta \cdot \exp \Bigl ( - \sum^n_{j=1} \frac {\lambda^2_j}{2} \Bigr )
$$
where $c_\beta > 0$  is an appropriate normalizing constant (see \cite{Dys} or
\cite{Mehta} among many others). It is well known that the empirical eigenvalue
distribution of $\frac {1}{\sqrt{n}}G_n$ converges weakly (almost surely) to
Wigner's semi-circle law. In the following discussion we will use the results
of Section \ref{subsection:N} to derive a corresponding result for a
$(2r+1)$-band matrix of a similar structure. To be precise consider the matrix
\begin{eqnarray*}
G^{(r)}_n= {\footnotesize {1 \over \sqrt2} \left (
\begin{array}{cccccccccc}
\sqrt{2} \ N_1 & \mathcal{X}_{(n-1)\gamma_1}
& \dots & \mathcal{X}_{(n-r)\gamma_r} & & & &  & \\
\mathcal{X}_{(n-1)\gamma_1} & \sqrt{2} \ N_2
& \dots  & \mathcal{X}_{(n-r)\gamma_{r-1}} & \mathcal{X}_{(n-r-1)\gamma_r} & & &  &\\
\mathcal{X}_{(n-2)\gamma_{2}} & \mathcal{X}_{(n-2)\gamma_{1}}
& \ddots & \ddots & \ddots  & \ddots & \ddots &   & \\
& \mathcal{X}_{(n-3)\gamma_{2}}
&  &\ddots & \ddots & \ddots  & \ddots & \ddots &   \\
 \vdots  &  \vdots
 &  \ddots  &\ddots & \ddots & \ddots  & \ddots & \ddots &   \\
\mathcal{X}_{(n-r)\gamma_r}  &  \mathcal{X}_{(n-r)\gamma_{r-1}}
 &   \ddots  &\ddots & \ddots & \ddots  & \ddots & \ddots &   \\
&  \mathcal{X}_{(n-r-1 )\gamma_r}   &
&   \ddots  &\ddots & \ddots & \ddots  & \ddots & \ddots &   \\
&   &
& \ddots   &\ddots & \ddots & \ddots  & \ddots & \ddots &   \\
&   &
  &\ddots & \ddots & \ddots  & \ddots & \ddots &   \\
& &
 & & & & & \mathcal{X}_{2 \gamma_1} & \mathcal{X}_{ \gamma_2} \\
& & 
 & & & & \mathcal{X}_{2 \gamma_1} & \sqrt{2}\ N_{n-1}  & \mathcal{X}_{ \gamma_1} \\
& & 
  & & &  &\mathcal{X}_{ \gamma_2} & \mathcal{X}_{\gamma_1} & \sqrt{2} \ N_n \\
 \end{array}
\right )}
\end{eqnarray*}
where $n=mr$,  all random variables in the matrix $G^{(r)}_n$ are independent
and $N_j$ is standard normal distributed ($j=1,\ldots , n$), while for
$j=1,\dots,n-1, \ k=1,\dots,r$ $\mathcal{X}^2_{j\gamma_k}$ has a chi-square
distribution with $j \gamma_k$ degrees of freedom $(\gamma_1, \dots, \gamma_r
\geq 0)$. It now follows by similar arguments as in \cite{Dette} that the
empirical distribution of the eigenvalues $\lambda^{(n,r)}_1 \leq \dots \leq
\lambda^{(n,r)}_n$ of the matrix $\frac {1}{\sqrt{n}} G^{(r)}_n$ has the same
asymptotic properties as the limiting distribution of the roots of the
orthogonal matrix polynomials $R_{m,n}(x)$ defined by $(R_{-1,n} (x)=0, \
R_{0,n}(x)=I_r)$
$$
x R_{k,n}(x)=A_{k+1,n} R_{k+1,n}(x)+B_{k,n}R_{k,n}(x)+A^*_{k,n} R_{k-1,n}(x); \quad k \geq 0,
$$
where the $r \times r$ matrices $A_{i,n}$ and $B_{i,n}$ are given by
\begin{eqnarray*}
A_{i,n} &=&{\footnotesize
\frac{1}{\sqrt{2n}}\left(
\begin{array}{cccccccc}
\sqrt{((i-1)r+1)\gamma_r} & 0  & 0  & \cdots   &  0      \\
\sqrt{((i-1)r+2)\gamma_{r-1}} &\sqrt{((i-1)r+2)\gamma_r}&  0  & \cdots &    0 \\
  &  &  &  &   \\
 \vdots &\ddots & \ddots & \ddots & \vdots  \\
  & & &  &   \\
 \sqrt{(ir-1)\gamma_2}  &\cdots  &  \sqrt{(ir-1)\gamma_{r-1} }   & \sqrt{(ir-1)\gamma_r} & 0 \\
 \sqrt{ir\gamma_1}  & \cdots&  \sqrt{ir\gamma_{r-2}}&\sqrt{ir\gamma_{r-1}} & \sqrt{ir\gamma_r} \\
 \end{array}
\right)}, \\
&& \\
{ B}_{i,n}&=&
{\footnotesize
\frac{1}{\sqrt{2n}}\left(
\begin{array}{cccccccc}
0 &  \sqrt{(ir+1)\gamma_1} &\sqrt{(ir+1)\gamma_2}  & \cdots   &   \sqrt{(ir+1)\gamma_{r-1}}       \\
\sqrt{(ir+1)\gamma_1}  &0& \sqrt{(ir+2)\gamma_1} & \cdots &    \sqrt{(ir+2)\gamma_{r-2}}  \\
  &  &  &  &   \\
 \vdots &\ddots & \ddots & \ddots & \vdots  \\
  & & &  &   \\
 \sqrt{(ir+1)\gamma_{r-2}}  &\cdots  &  \sqrt{((i+1)r-2)\gamma_{1}}   &  0 & \sqrt{((i+1)r-2)\gamma_1} \\
 \sqrt{(ir+1)\gamma_{r-1} }  & \cdots&  \sqrt{((i+1)r-2)\gamma_{2}} & \sqrt{((i+1)r-1)\gamma_1} & 0 \\
 \end{array}
\right)},
\end{eqnarray*}
Now it is easy to see that for any $u \in (0,1)$
\begin{eqnarray*}
\lim_{\frac{i}{n}\to u} {B}_{i,n}= B(u) , ~ \lim_{\frac{i}{n}\to u} {A}_{i,n}=
A (u),
\end{eqnarray*}
where
\begin{eqnarray}\label{0.1}
A (u)&:=&\sqrt{\frac{ur}{2}}\left(
\begin{array}{cccccccc}
\sqrt{\gamma_r} &  0  & 0  & \cdots   & 0      \\
\sqrt{\gamma_{r-1}} & \sqrt{\gamma_r}&  0  & \cdots &   0 \\
  &  &  &  &   \\
 \vdots &\ddots & \ddots & \ddots & \vdots  \\
  & & &  &   \\
 \sqrt{\gamma_2}  &\cdots  &  \sqrt{\gamma_{r-1} }   & \sqrt{\gamma_r} & 0 \\
 \sqrt{\gamma_1}  & \cdots&  \sqrt{\gamma_{r-2}}&\sqrt{\gamma_{r-1}} & \sqrt{\gamma_r} \\
 \end{array}
\right)\in\er^{r\times r}, \\
\nonumber  \\ \label{0.2} B(u)&:=&\sqrt{\frac{ur}{2}}\left(
\begin{array}{cccccccc}
0 &  \sqrt{\gamma_1} &\sqrt{\gamma_2}  & \cdots   &   \sqrt{\gamma_{r-1}}       \\
\sqrt{\gamma_1}  &0& \sqrt{\gamma_1} & \cdots &    \sqrt{\gamma_{r-2}}  \\
  &  &  &  &   \\
 \vdots &\ddots & \ddots & \ddots & \vdots  \\
  & & &  &   \\
 \sqrt{\gamma_{r-2}}  &\cdots  &  \sqrt{\gamma_{1}}   &  0 & \sqrt{\gamma_1} \\
 \sqrt{\gamma_{r-1} }  & \cdots&  \sqrt{\gamma_{2}} & \sqrt{\gamma_1} & 0 \\
 \end{array}
\right)\in\er^{r\times r},
 \end{eqnarray}
and Theorem \ref{theorem:zerodistr:N} shows that the normalized counting
measure of the roots of the matrix orthogonal polynomial $R_{m,n}(x)$ has a
weak limit $\mu_{0,1/r}$ defined by its logarithmic potential, that is
$$
\int \log |x-t|^{-1} d \mu_{0,1/r}(t)= \int^{1/r}_0 \log | z_1(x,u), \dots,
z_r(x,u) | du + c(r)
$$
where $z_1(x,u),\dots,z_r(x,u)$ are the roots of the equation
\begin{equation} \label{eq1}
\det (A^*(u)z+B(u)+A(u) z^{-1} - x I_r)=0
\end{equation}
corresponding to the smallest moduli. Observing the structure of the matrices in (\ref{0.1}) and (\ref{0.2}) it follows that
$$
z_j(x,u)=z_j \Bigl ( \frac {x}{\sqrt{u}} \Bigr )
$$
where $z_1(x),\dots,z_r(x)$ are the roots of the equation (\ref{eq1}) for $u=1$
(with smallest modulus). Similar arguments as given in the proof of Proposition
\ref{prop:mu0:bis} now show that the measure $\mu_{0,1/r}$ is absolutely
continuous with density defined by
\begin{eqnarray} \label{lim}
\frac {\ud \mu_{0,1/r}(x)}{\ud x} &=& \int^{1/r}_0 \frac {1}{2 \pi} \sum_{k:
|z_k(x,u)|=1} \Big | \frac{{\partial \over \partial x } z_k(x,u)}{ z_k (x,u)}
\Big | \ud u \\ \label{lim:2} &=& \frac {1}{2 \pi} \int^{1/r}_0 \frac
{1}{\sqrt{u}} \sum_{k: |z_k(x/\sqrt{u})|=1} \Big |  \frac
{z_k^\prime(x/\sqrt{u})}{z_k (x/\sqrt{u})} \Big | \ud u.
\end{eqnarray}
By the same reasoning as in
\cite{Dette}  we therefore obtain the following result ($\delta_x$ denotes the
Dirac measure).

\begin{theorem}\label{theoranmatrix}
 If $\lambda_1^{(n,r)} \leq \dots \leq \lambda_n^{(n,r)}$ denote the eigenvalues of the random matrix $G_n^{(r)}/\sqrt{n}$ with $\gamma_1,\dots,\gamma_r > 0$, then the empirical eigenvalue distribution
$$
 \frac {1}{n} \sum^n_{j=1} \delta_{\lambda_j^{(n,r)}}
$$
converges weakly (almost surely) to the measure $\mu_{0,1/r}$ defined in
(\ref{lim}).
\end{theorem}

We conclude this section with a brief example illustrating Theorem
\ref{theoranmatrix} in the case $r=2$. In the left part of Figure \ref{fig1} we
show the simulated eigenvalue distribution of the matrix $G_n^{(r)}/\sqrt{n}$
in the case $n=5000$ $\gamma_1=\gamma_2=1$, while the corresponding limiting
density is shown in the right part of the figure. Similar results in the case
$r=2$, $\gamma_1=1$ and $\gamma_2=5$ are depicted in Figure \ref{fig2}.
{Note that the derivatives $z_k'(x)$ can be evaluated numerically using
the formula (implicit function theorem)}
$$ z_k'(x) = -\frac{\frac{\partial f}{\partial x}(z_k(x),x)}{\frac{\partial f}{\partial z}(z_k(x),x)}.
$$

\begin{figure}\begin{center}
\includegraphics[width=7cm]{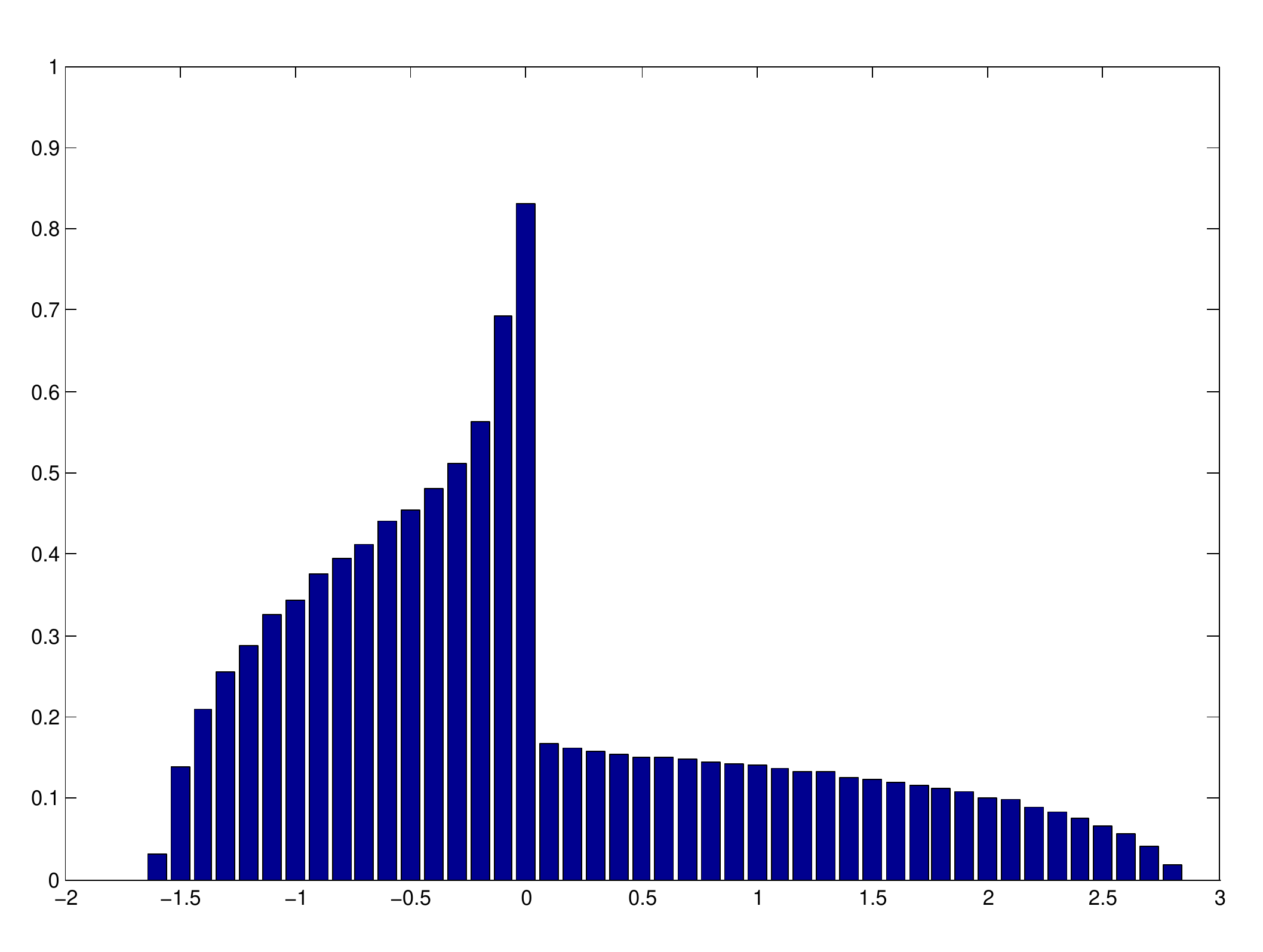}~~
 \includegraphics[width=7cm]{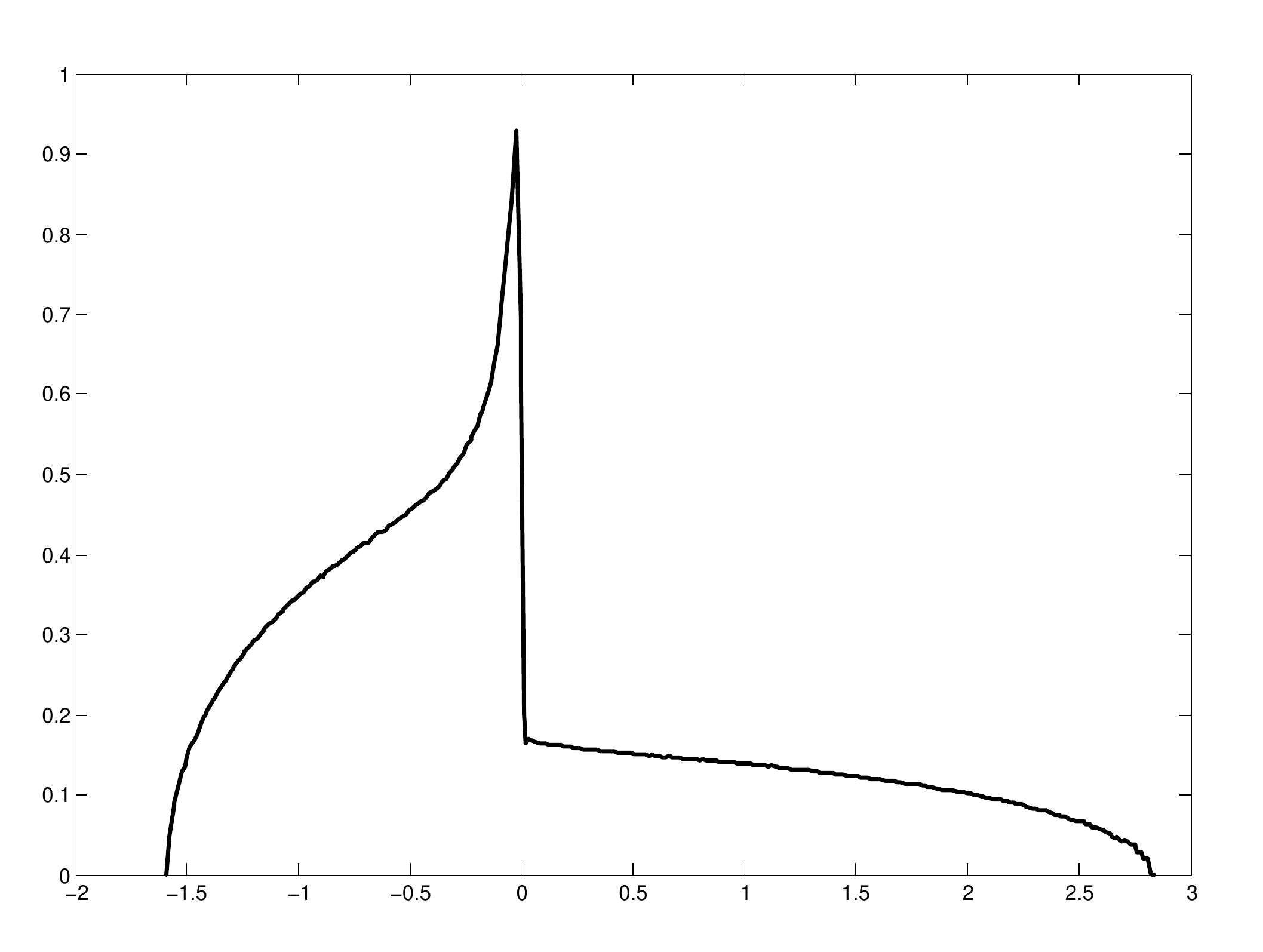}
\end{center}
\caption{
 \label{fig1}
\it Simulated and limiting spectral density of the random block matrix
$G_n^{(r)}/\sqrt{n}$ in the case $r=2$, $\gamma_1=1$, $\gamma_2=1$. In the
simulation the eigenvalue distribution of a $5000 \times 5000$ matrix was
calculated (i.e. m = n/r = 2500).}
\end{figure}

\begin{figure}
\begin{center}
\includegraphics[width=7cm]{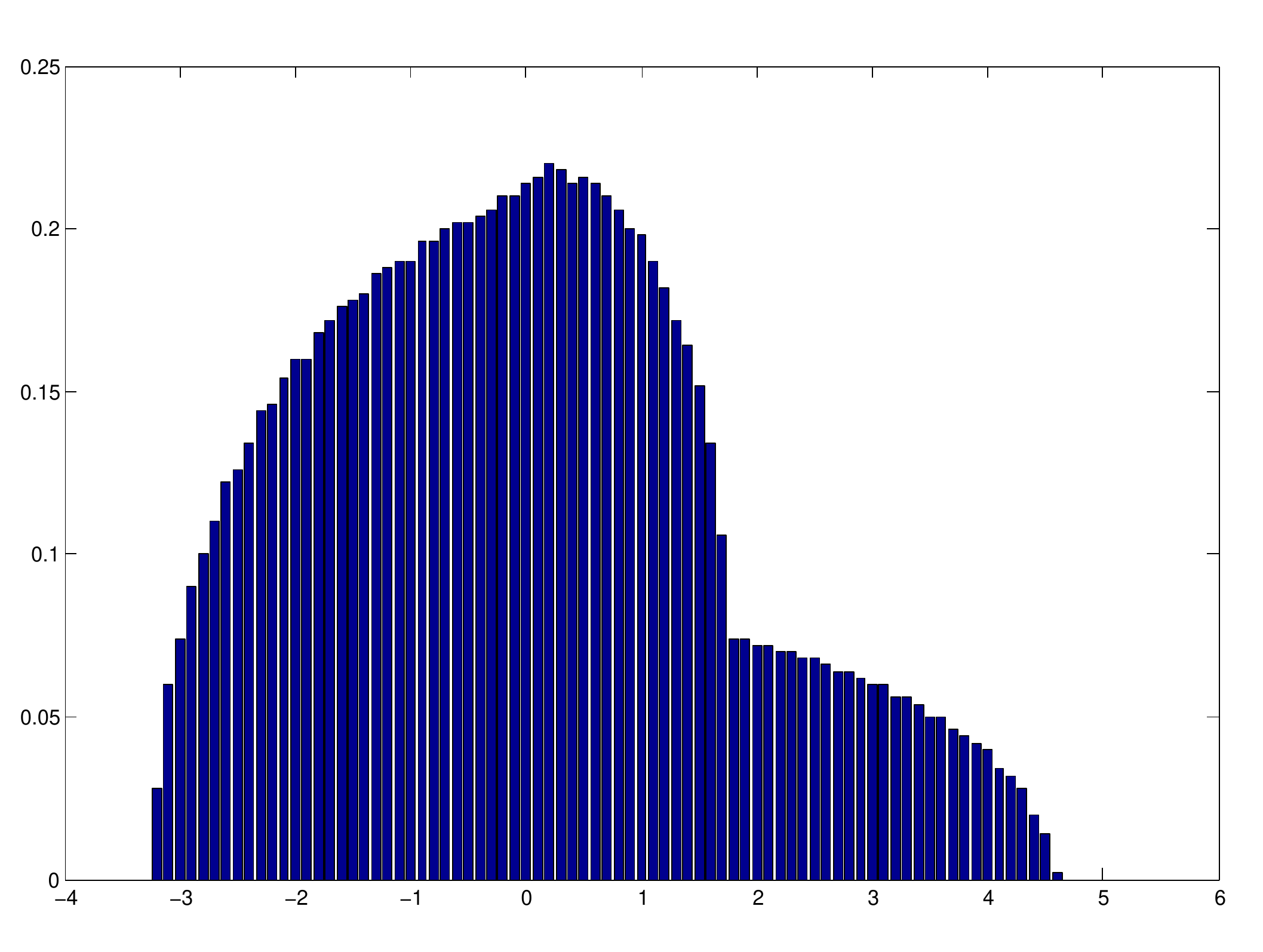}~~
 \includegraphics[width=7cm]{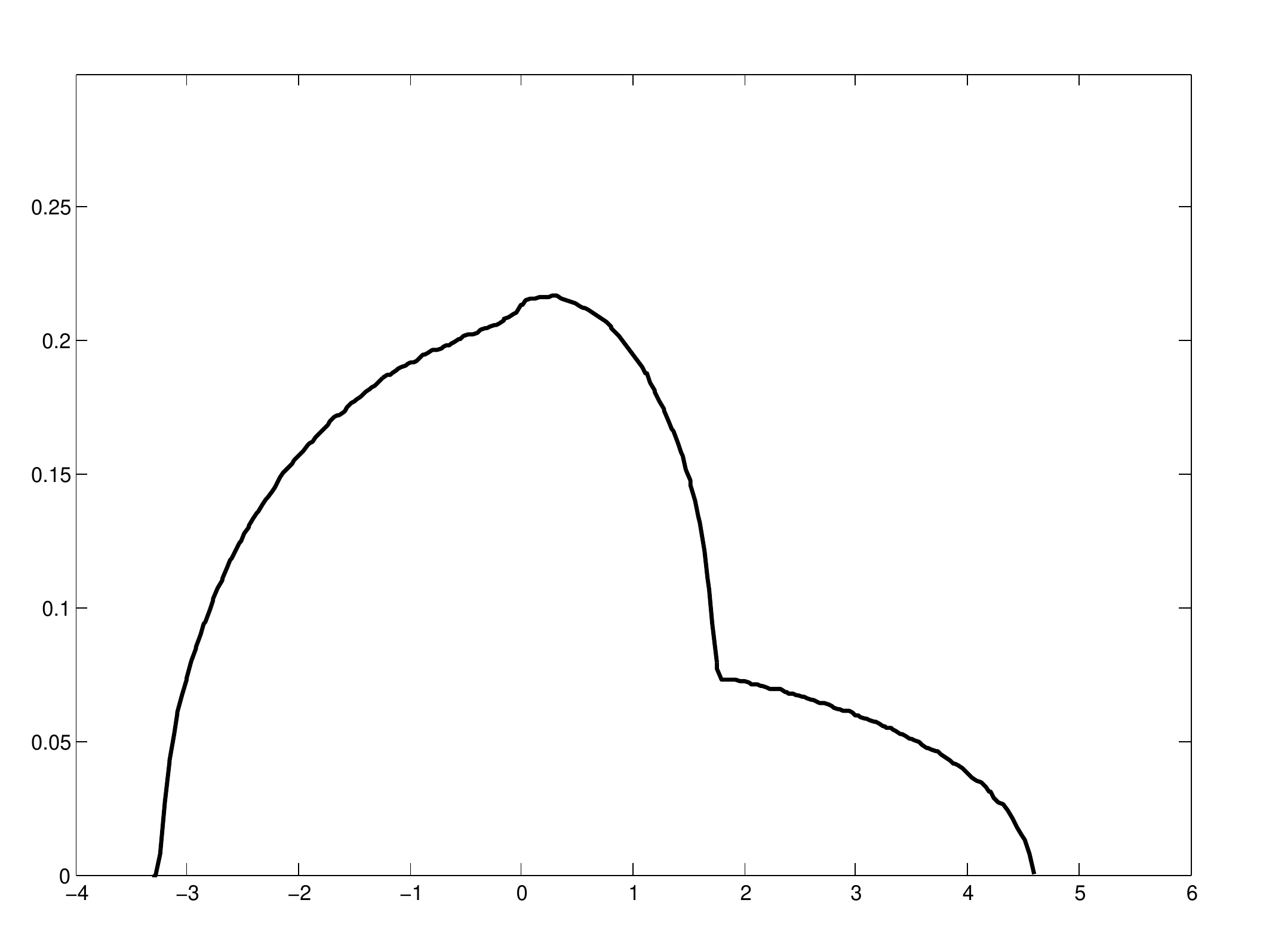}
\end{center}
\caption{ \label{fig2} \it Simulated and limiting spectral density of the
random block matrix $G_n^{(r)}/\sqrt{n}$ in the case $r=2$, $\gamma_1=1$,
$\gamma_2=5$. In the simulation the eigenvalue distribution of a $5000 \times
5000$ matrix was calculated (i.e. m = n/r = 2500).}
\end{figure}

\section{Proofs}\label{section:proofs}

\subsection{Proof of Proposition~\ref{prop:Gamma0}}
\label{subsection:proof:Gamma0}

First we consider the behavior of the roots $z_k(x)$ for $x\to\infty$. It is
easy to check that in this limit,
\begin{equation}\label{roots:infinity}
\left\{\begin{array}{ll} z_k(x)\to 0,& \qquad k=1,\ldots,r,\\
|z_k(x)|\to\infty,& \qquad k=r+1,\ldots,2r,\end{array}\right.
\end{equation}
see e.g.\ \cite{Delvaux}. In particular we have that $|z_r(x)|<|z_{r+1}(x)|$ if
$|x|$ is large enough, showing the compactness of $\Gamma_0$.

Next we ask the question: for which $x\in\cee$ can we have a root $z_k=z_k(x)$
such that $|z_k|=1$? In this case, \eqref{algebraic:equation} becomes
$$ 0 = f(z_k,x) = \det(z_k A^*+B+\overline{z_k}A-x I_r),
$$
so $x$ is an eigenvalue of the Hermitian matrix $z_k A^*+B+\overline{z_k}A$. In
particular, it follows that
\begin{equation}\label{root:unitcircle}
|z_k(x)|=1\ \Rightarrow\ x\in\er.
\end{equation}
Now for $x\to\infty$, \eqref{roots:infinity} implies that we have precisely $r$
roots $z_k(x)$ with $|z_k(x)|<1$. By \eqref{root:unitcircle} and continuity
this must then hold for all $x\in\cee\setminus\er$:
$$ x\in\cee\setminus\er\ \Rightarrow\ \left\{\begin{array}{ll}
|z_k(x)|<1,&\quad k=1,\ldots,r,\\ |z_k(x)|>1,&\quad
k=r+1,\ldots,2r.\end{array}\right.
$$
In particular we see that $|z_r(x)|<1<|z_{r+1}(x)|$ for $x\in\cee\setminus\er$,
implying that $\Gamma_0\subset\er$.

Finally, the claim that $\Gamma_0\subset\er$ is the disjoint union of at most
$r$ intervals follows from \cite{Delvaux,Widom1}. $\bol$

\subsection{Proof of Lemma~\ref{lemma:alggeomult}}
\label{subsection:proof:alggeo}

The proof will use ideas from Dur\'an \cite[Proof of Lemma~2.2]{Duran}. We
invoke a general result known as the \emph{square-free factorization for
multivariate polynomials}. In our context it implies that there exists a
factorization of the bivariate polynomial $z^r f(z,x)$ of the form
\begin{equation}\label{euclid:1} z^r f(z,x) = \prod_{k=1}^{K} g_k(z,x)^{m_k},
\end{equation}
for certain $K\in\zet_{>0}$, multiplicities $m_1,\ldots,m_K\in \zet_{> 0}$ and
non-constant bivariate polynomials $g_1(z,x),\ldots,g_K(z,x)$, in such a way
that
\begin{equation}\label{euclid:2} g(z,x):=\prod_{k=1}^{K} g_k(z,x)
\end{equation}
is square-free, i.e., for all but finitely many $x\in\cee$, the roots $z$ of
$g(z,x)=0$ are all distinct, and vice versa with the roles of $x$ and $z$
reversed.

The existence of the square-free factorization of the previous paragraph, can
be obtained from a repeated use of Euclid's algorithm. For example, if $z^r
f(z,x)=0$ has a multiple root $z=z(x)$ for all $x\in\cee$, then we apply
Euclid's algorithm with input polynomials $z^r f(z,x)$ and
$\frac{\partial}{\partial z}\left(z^r f(z,x)\right)$, viewed as polynomials in
$z$ with coefficients in $\cee[x]$. This gives us the greatest common divisor
of these two polynomials and yields a factorization
$$ z^r f(z,x) = g_1(z,x)g_2(z,x),
$$
for two bivariate polynomials $g_1,g_2$ which depend both nontrivially on $z$.
{Note that the factorization can be taken fraction free, i.e., with
$g_1,g_2$ being polynomials in $x$ rather than rational functions.} If one of
the factors $g_1$ or $g_2$ has a multiple root $z=z(x)$ for all $x\in\cee$,
then we repeat the above procedure. If $g_1$ and $g_2$ have a common root
$z=z(x)$ for all $x\in\cee$, then we apply Euclid's algorithm with input
polynomials $g_1$ and $g_2$, viewed again as polynomials in $z$ with
coefficients in $\cee[x]$. Repeating this procedure sufficiently many times
yields the square-free factorization in the required form.


Note that the factors $g_k(z,x)$ in \eqref{euclid:1} all depend non-trivially
on both $z$ and $x$. For if $g_k(z,x)$ would be a polynomial in $z$ alone
(say), then there would exist $z\in\cee$ such that $f(z,x)=0$ for all
$x\in\cee$, which is easily seen to contradict with \eqref{algebraic:equation}.

From the above paragraphs we easily get the symmetry relation
$$m_1(z,x) = m_2(z,x),
$$
for all but finitely many $z,x\in\cee$. This proves one part of
Lemma~\ref{lemma:alggeomult}.

It remains to show that $m_2(z,x)=d(z,x)$ for all but finitely many
$z,x\in\cee$. From the definitions, this is equivalent to showing that the
matrix $zA^*+B+z^{-1}A$ is diagonalizable for all but finitely many $z\in\cee$.
This is certainly true if $|z|=1$ since then $zA^*+B+z^{-1}A$ is Hermitian so
in particular diagonalizable. The claim then follows in exactly the same way as
in \cite[Proof of Lemma~2.2]{Duran}. $\bol$

\subsection{Proof of Theorem~\ref{theorem:ratioasy}}
\label{subsection:proof:ratio}

Before going to the proof of Theorem~\ref{theorem:ratioasy}, let us recall the
following result of Dette-Reuther \cite{Dette}. As mentioned before, this
result uses in an essential way the Hermitian structure of
\eqref{block:Jacobi}.
\begin{lemma}\label{lemma:DetteReuther} (See \cite{Dette}:)
If all roots of the matrix orthogonal polynomials $P_n(x)$ are located in the
interval $[-M,M]$, then the inequality $$ |\vecv^*
P_{n}(z)P^{-1}_{n+1}(z)A^{-1}_{n+1}\vecv| \leq
\frac{1}{\mathrm{dist}(z,[-M,M])}$$ holds for all complex numbers $z$ and for
all column vectors $\vecv$ with unit Euclidean norm $||\vecv||=1$.
\end{lemma}

We will need the following variant of Lemma~\ref{lemma:DetteReuther}:

\begin{corollary} In the matrix Nevai class \eqref{Nevai:class} we can find
$M>0$ as in Lemma~\ref{lemma:DetteReuther} so that
\begin{equation}\label{polar:identity} |\vecw^* P_{n}(z)P^{-1}_{n+1}(z)\vecv|
\leq \frac{8||A||}{\mathrm{dist}(z,[-M,M])}
\end{equation}
for all $n$ sufficiently large, all column vectors $\vecv,\vecw$ with
$||\vecv||=||\vecw||=1$ and all complex numbers $z$. Here $||A||$ denotes the
operator norm (also known as $2$-norm, or maximal singular value) of the matrix
$A$.
\end{corollary}

\begin{proof}
For fixed $n$ and fixed $z\in\cee$, define the sesquilinear form $$\langle
\vecv,\vecw\rangle_A := \vecw^* P_{n}(z)P^{-1}_{n+1}(z)A^{-1}_{n+1}\vecv,$$
which is linear in its first and antilinear in its second argument. Also define
$||\vecv||_A^2 := \langle\vecv,\vecv\rangle_A$. (This \lq norm\rq\ is not
necessarily positive!) The polar identity asserts that
$$\langle \vecv,
\vecw\rangle_A = \frac 14 \left(||\vecv + \vecw||_A^2 - ||\vecv - \vecw||_A^2 +
i||\vecv + i\vecw||_A^2 - i||\vecv - i\vecw||_A^2\right).$$ Combining this with
Lemma~\ref{lemma:DetteReuther} we get
$$ |\vecw^*
P_{n}(z)P^{-1}_{n+1}(z)A^{-1}_{n+1}\vecv|=:|\langle \vecv,\vecw\rangle_A|  \leq
\frac{4}{\mathrm{dist}(z,[-M,M])}$$ for all pairs of vectors $\vecv,\vecw$ with
$||\vecv||=||\vecw||=1$ and for all complex numbers $z$. If we now take $n$
sufficiently large so that $||A_{n+1}||<2||A||$ (recall \eqref{Nevai:class}),
we get the desired inequality \eqref{polar:identity}.
\end{proof}

\begin{proof}[Proof of Theorem~\ref{theorem:ratioasy}]
We will use a normal family argument. Fix $k\in\{1,\ldots,r\}$ and
$z_0\in\cee\setminus [-M,M]$. By \eqref{polar:identity} we have that
$P_{n}(z)P^{-1}_{n+1}(z)\vecv_k(z)$ is uniformly bounded entrywise in a
neighborhood of $z=\infty$. {By Montel's theorem we can take} a
subsequence of indices $(n_i)_{i=0}^{\infty}$ so that the limit
$\lim_{i\to\infty} P_{n_i}(z)P^{-1}_{n_i+1}(z)\vecv_k(z)$ exists uniformly in
this neighborhood. We will prove by induction on $l=0,1,2,\ldots$ that
\begin{equation}\label{induction:hypothesis}
\lim_{i\to\infty}P_{n_i}(x)P_{n_i+1}^{-1}(x)\vecv_k(x) =
z_k(x)\vecv_k(x)+O(x^{-l}),\qquad x\to\infty.
\end{equation}
For $l=0$ (or $l=1$) this follows from \eqref{polar:identity} and the fact that
$z_k(x)=O(x^{-1})$ as $x\to\infty$.

Now assume that  \eqref{induction:hypothesis} is satisfied for a certain value
of $l$, and for \emph{any} sequence $(n_i)_i$ for which the limit exists.
Fixing such a sequence $(n_i)_i$, we will prove that
\eqref{induction:hypothesis} holds with $l$ replaced by $l+2$. By moving to a
subsequence of $(n_i)_i$ if necessary, we may assume without loss of generality
that the limiting matrices $\lim_{i\to\infty}P_{n_i}(x)P_{n_i+1}^{-1}(x)$ and
$\lim_{i\to\infty}P_{n_i-1}(x)P_{n_i}^{-1}(x)$ both exist. The induction
hypothesis asserts that \eqref{induction:hypothesis} holds for \emph{any}
sequence $(n_i)$ for which the limit exists, so in particular
\begin{equation}\label{induction:hypothesis:shifted}
\lim_{i\to\infty}P_{n_i-1}(x)P_{n_i}^{-1}(x)\vecv_k(x) =
z_k(x)\vecv_k(x)+O(x^{-l}),\qquad x\to\infty.
\end{equation}

Now let us write the three-term recurrence \eqref{recurrence} in the form
$$ A_{n_i}^{*}P_{n_i-1}(x)P_{n_i}^{-1}(x) + (B_{n_i}-xI_r)+A_{n_i+1}
P_{n_i+1}(x)P_{n_i}^{-1}(x)=0.
$$
Multiplying on the right with $\vecv_k(x)$, taking the limit $i\to\infty$ and
using the facts that $\lim_n A_n=A$ and $\lim_n B_n=B$ we get
$$ A^{*}\lim_{i\to\infty}
\left(P_{n_i-1}(x)P_{n_i}^{-1}(x)\right)\vecv_k(x) +
(B-xI_r)\vecv_k(x)+A\lim_{i\to\infty}
\left(P_{n_i+1}(x)P_{n_i}^{-1}(x)\right)\vecv_k(x)=0.
$$
With the help of \eqref{induction:hypothesis:shifted} and \eqref{vector:v} this
implies
$$ A\lim_{i\to\infty}\left(P_{n_i+1}(x)P_{n_i}^{-1}(x)\right)\vecv_k(x) =
A z_k^{-1}(x)\vecv_k(x) + O(x^{-l}),\qquad x\to\infty.
$$
Multiplying this relation on the left with
$z_k(x)\lim_{n\to\infty}\left(P_{n_i}(x)P_{n_i+1}^{-1}(x)\right)A^{-1}$ and
using \eqref{polar:identity} and the fact that $z_k(x)=O(x^{-1})$ for
$x\to\infty$, we find
$$ \lim_{i\to\infty}P_{n_i}(x)P_{n_i+1}^{-1}(x)\vecv_k(x) =
z_k(x)\vecv_k(x)+O(x^{-l-2}),\qquad x\to\infty,
$$
showing that the induction hypothesis \eqref{induction:hypothesis} holds with
$l$ replaced with $l+2$. This proves the induction step.
\end{proof}

\subsection{Proof of Theorem~\ref{theorem:zerodistr}}
\label{subsection:proof:zero}

We write the telescoping product
$$ \det P_n(x) = \left(\det P_n(x)P_{n-1}^{-1}(x)\right)\ldots \left(\det P_2(x)P_{1}^{-1}(x)\right)\left(\det
P_1(x)P_{0}^{-1}(x) \right)\det P_0(x).$$ Taking logarithms and dividing by
$rn$ we get
$$ \frac{1}{rn} \log\det P_n(x) = \frac{1}{rn}\left(\left(\sum_{k=1}^{n} \log\det
P_k(x)P_{k-1}^{-1}(x)\right)+\log\det P_0(x)\right).
$$
(Here we use the logarithm as a complex multi-valued function.) Taking the
limit $n\to\infty$ and using the ratio asymptotics in
Corollary~\ref{cor:ratiodet}, we obtain
$$ -\lim_{n\to\infty}
\frac{1}{rn} \log(\det P_n(x)) = \frac 1r\log (z_1(x)\ldots z_r(x)),\qquad
x\in\cee\setminus [-M,M].
$$
Now we take the real part of both sides of this equation. Then the left hand
side becomes precisely the logarithmic potential of $\mu_0$, up to an additive
constant $C$. So we obtain \eqref{limiting:measure:mu0}; the constant $C$ can
be determined by calculating the asymptotics for $x\to\infty$. $\bol$

\subsection{Proof of Proposition~\ref{prop:Gamma0:bis}}
\label{subsection:proof:Gamma0bis}

From Proposition~\ref{prop:Gamma0} and its proof we know that both the left and
right hand sides of the equation in \eqref{Gamma0:unitnorm} are subsets of the
real axis. Now for $x\in\er$ one has the symmetry relation \begin{multline}
\overline{f(z,x)} := \overline{\det(A^*z+B+Az^{-1}-x I_r)} = \det(A \bar
z+B+A^*\bar z^{-1}-x I_r) =  f(\bar{z}^{-1},x),
\end{multline}
where the bar denotes the complex conjugation and we used that $\overline{\det
M}=\det(M^*)$ for any square matrix $M$. This implies that for each solution
$z=z_k(x)$ of the equation $f(z,x)=0$, the complex conjugated inverse $z=\bar
z_k^{-1}(x)$ is a solution as well, with the same multiplicity. So with the
ordering \eqref{rootszk:ordering} we have that $|z_k(x)|\cdot |z_{2r-k}(x)|=1$
for any $k=1,\ldots,r$. In particular we have that $|z_r(x)|=|z_{r+1}(x)|$ if
and only if $|z_r(x)|=1$. This implies Proposition~\ref{prop:Gamma0:bis}.
$\bol$ 

\subsection{Proof of Proposition~\ref{prop:mu0:bis}}
\label{subsection:proof:mu0bis}

{We use the notations in Section~\ref{subsection:mu0alternative}. We fix $x\in
\mathcal I\subset\Gamma_0$ and define the sets}
\begin{align}\label{Splus}
S_+(x) = \{ z_{k}(x)\mid |z_k(y)|<1\textrm{ for all }y\in \Omega\cap\cee_+\},\\
\label{Sminus} S_-(x) = \{ z_{k}(x)\mid |z_k(y)|<1\textrm{ for all }y\in
\Omega\cap\cee_-\}.
\end{align}
{So $S_+(x)$ (or $S_-(x)$) contains all the roots $z_k(x)$ for which
$|z_k(y)|<1$ for $y$ in the upper half plane (or lower half plane respectively)
close to $x$.}

Let $z_k(x)$ be a root of modulus strictly less than $1$. By continuity this
root belongs to both sets $S_+(x)$ and $S_-(x)$,
with the same multiplicity, 
and hence the contributions from the $+$- and $-$-terms in \eqref{measure:mu0}
corresponding to this root $z_k(x)$ cancel out.

Next let $z_k(x)$ be a root of modulus $1$. Assume again that
$z_k(y)=e^{i\theta_k(y)}$ with $\theta_k(y)$ {real and differentiable} for
$y\in \mathcal I\subset\er$. Suppose that $\theta_k'(x)>0$. Then the
Cauchy-Riemann equations applied to $\log z_k(y)$ imply that $|z_k(y)|<1$ for
$y$ in the upper half plane close to $x$, and $|z_k(y)|>1$ for $y$ in the lower
half plane close to $x$. So $z_k(x)$ lies in the set $S_+(x)$ in \eqref{Splus}
but not in $S_-(x)$. Similarly if $\theta_k'(x)<0$ then $z_k(x)$ lies in the
set $S_-(x)$ but not in $S_+(x)$. In both cases, the contribution from $z_k(x)$
in the right hand side of \eqref{measure:mu0} has a positive sign and so we
obtain the desired equality \eqref{measure:mu0:bis}.

Finally, the claim that $\theta'(x)\neq 0$ for any $x\in \mathcal
I\subset\Gamma_0$ follows since, if this fails, then general considerations
(e.g.\ in \cite[Proof of Theorem 11.1.1(v)]{Simon2}) would imply that
$\Gamma_0\not\subset\er$, which is a contradiction.
$\bol$ 

\subsection{Proof of Theorem~\ref{theorem:ratioasy:N}}
\label{subsection:proof:N}

The proof of Theorem~\ref{theorem:ratioasy}  and \ref{theorem:zerodistr} can be easily extended to prove
Theorem~\ref{theorem:ratioasy:N}. The difference is that the limits
$\lim_{n_i\to\infty}$ should be replaced by local limits of the form
$\lim_{n_i/N\to s}$. The details are straightforward and left to the reader
(for similar reasonings see also \cite{BDK,CCV,Dette,Roman,KVA}, among others.)

\subsection{Proof of Theorem~\ref{theorem:ratioasy:p}}
\label{subsection:proof:p}



The proof of Theorem~\ref{theorem:ratioasy:p} will follow the same scheme as
the proof of Theorem~\ref{theorem:ratioasy}, but it will be more complicated
due to the higher periodicity. To deal with the periodicity we will use some
ideas from \cite{BDK}. It is convenient to substitute $z=y^{p}$ and work with
the transformed matrix
\begin{eqnarray} \label{Gyx} G(y,x) &:=&
\diag(1,y,\ldots,y^{p-1})F(y^p,x)\diag(1,y^{-1},\ldots,y^{-(p-1)})\\
\label{Gyx:bis}&=& \begin{pmatrix} B^{(0)}-xI_r & y^{-1} A^{(1)} & 0 & 0 & y A^{(0)*} \\
y A^{(1)*} & B^{(1)}-xI_r & y^{-1} A^{(2)} & 0 & 0 \\
0 & y A^{(2)*} & \ddots & \ddots & 0 \\
0 & 0 & \ddots & \ddots & y^{-1} A^{(p-1)}\\
y^{-1}A^{(0)} & 0 & 0 & y A^{(p-1)*} & B^{(p-1)}-xI_r
\end{pmatrix}_{pr\times pr}.
\end{eqnarray}
Consistently with the substitution $z=y^p$, we put $y_k(x)=z_k^{1/p}(x)$,
$k=1,\ldots,2r$, for an arbitrary but fixed choice of the $p$th root. The
ordering \eqref{rootszk:ordering} implies that
\begin{equation}\label{rootsyk:ordering}
0<|y_1(x)|\leq |y_2(x)|\leq\ldots\leq |y_r(x)|\leq |y_{r+1}(x)|\leq\ldots\leq
|y_{2r}(x)|.
\end{equation}
Note that each $y=y_k(x)$ is a root of the algebraic equation
$$\det
G(y,x)\equiv\det F(y^p,x)=0.
$$
 From \eqref{Gyx:bis} it is then easy to check
that (see e.g.~\cite{Delvaux})
\begin{equation}\label{rootsyk:asy} y_k(x)\propto \left\{
\begin{array}{lll} x^{-1},& k=1,\ldots,r,& x\to\infty, \\
x,& k=r+1,\ldots,2r,& x\to\infty,\end{array}\right. \end{equation} where the
$\propto$ symbol means that the ratio of the left and right hand sides is
bounded both from below and above in absolute value when $x\to\infty$.

Denote with $\vecw_k(x)$ a normalized null space vector such that
\begin{equation}\label{vector:w}
G(y_k(x),x)\vecw_k(x) = \mathbf{0}.
\end{equation}
If there are roots $y_k(x)$ with higher multiplicities then we pick the vectors
$\vecw_k(x)$ as explained in Section~\ref{subsection:ratioasy}. We again
partition $\vecw_k(x)$ in blocks as
\begin{equation}\label{nullspacevector:bis} \vecw_{k}(x) =
\begin{pmatrix}
\vecw_{k,0}(x)\\ \vdots \\ \vecw_{k,p-1}(x)
\end{pmatrix},
\end{equation}
where each $\vecw_{k,j}(x)$, $j=0,1,\ldots,p-1$, is a column vector of length
$r$. Assuming the normalization $||\vecw_{k}(x)||=1$ then we have that
\begin{equation}\label{wk:boundednorm} \lim_{x\to\infty}
||\vecw_{k,j}(x)|| = C_{k,j}>0,\qquad j=0,1,\ldots,p-1.
\end{equation}
This follows from \eqref{vector:w}--\eqref{nullspacevector:bis},
\eqref{rootsyk:asy} and by inspecting the dominant terms for $x\to\infty$ in
the matrix \eqref{Gyx:bis}.

Theorem~\ref{theorem:ratioasy:p} will be a consequence of the following
stronger statement:
\begin{equation}\label{ratio:asy:p:strong}
\left(\lim_{n\to\infty}P_{pn+j}(x)P_{pn+j+1}^{-1}(x)\right)\vecw_{k,j+1}(x) =
y_k(x)\vecw_{k,j}(x),\quad x\to\infty,
\end{equation}
uniformly for $x$ in compact subsets of $\cee\setminus ([-M,M]\cup S)$, for all
$k\in\{1,\ldots,r\}$ and for all residue classes $j\in\{0,1,\ldots,p-1\}$
modulo $p$. (We identify $\vecw_{k,p}(x)\equiv \vecw_{k,0}(x)$.) Indeed,
Theorem~\ref{theorem:ratioasy:p} immediately follows by iterating
\eqref{ratio:asy:p:strong} $p$ times and using that $y_k^p(x)=z_k(x)$.
\smallskip

The rest of the proof is devoted to establishing \eqref{ratio:asy:p:strong}. We
will show by induction on $l\geq 0$ that
\begin{equation}\label{induction:hypothesis:p}
\left(\lim_{i\to\infty}P_{pn_i+j}(x)P_{pn_i+j+1}^{-1}(x)\right)\vecw_{k,j+1}(x)
= y_k(x)\vecw_{k,j}(x)(1+O(x^{-l})),\quad x\to\infty,
\end{equation}
for any $k\in\{1,\ldots,r\}$ and $j\in\{0,1,\ldots,p-1\}$, and for any
increasing sequence $\left(n_i\right)_{i=0}^{\infty}$ for which the limit in
the left hand side exists.

Assume that the induction hypothesis \eqref{induction:hypothesis:p} holds for a
certain value of $l\geq 0$. We will show that it also holds for $l+2$. Let
$\left(n_i\right)_{i=0}^{\infty}$ be an increasing sequence for which the limit
in the left hand side of \eqref{induction:hypothesis:p} exists. We can assume
without loss of generality that $j=p-1$; a similar argument will work for the
other values of $j\in\{0,1,\ldots,p-1\}$. Now from the three-term recursion we
obtain
\begin{equation}
\begin{pmatrix}
A_{pn_i}^* & B_{pn_i}-xI_r & A_{pn_i+1} \\
& A_{pn_i+1}^* & \ddots & \ddots  \\
& & \ddots & \ddots & A_{pn_i+p-1}\\
& & & A_{pn_i+p-1}^* & B_{pn_i+p-1}-xI_r & A_{pn_i+p}
\end{pmatrix}
\begin{pmatrix}
P_{pn_i-1}(x)\\ P_{pn_i}(x) \\ \vdots \\ P_{pn_i+p-1}(x)\\ P_{pn_i+p}(x)
\end{pmatrix}=0.
\end{equation}
Applying a diagonal multiplication with appropriate powers of $y:=y_k(x)$ we
get
\begin{equation}\label{aux:product}
\begin{pmatrix}
yA_{pn_i}^* & B_{pn_i}-xI_r & y^{-1}A_{pn_i+1} \\
& yA_{pn_i+1}^* & \ddots & \ddots  \\
& & \ddots & \ddots & y^{-1}A_{pn_i+p-1}\\
& & & yA_{pn_i+p-1}^* & B_{pn_i+p-1}-xI_r & y^{-1}A_{pn_i+p}
\end{pmatrix}
\begin{pmatrix}
y^{-p}P_{pn_i-1}(x)\\ y^{-(p-1)}P_{pn_i}(x) \\ \vdots \\ y^{-1}P_{pn_i+p-2}(x)\\ P_{pn_i+p-1}(x)\\
yP_{pn_i+p}(x)
\end{pmatrix}=0.
\end{equation}
Let us focus on the rightmost matrix in the left hand side of
\eqref{aux:product}. Multiplying on the right with
$P_{pn_i+p-1}^{-1}(x)\vecw_{k,p-1}(x)$ it becomes
\begin{equation}\label{aux:vector}
\begin{pmatrix} y^{-p}P_{pn_i-1}P_{pn_i+p-1}^{-1}\vecw_{k,p-1} \\
\vdots \\
y^{-1}P_{pn_i+p-2}P_{pn_i+p-1}^{-1}\vecw_{k,p-1} \\ \vecw_{k,p-1} \\
yP_{pn_i+p}P_{pn_i+p-1}^{-1}\vecw_{k,p-1}
\end{pmatrix}.
\end{equation}
(Here we skip the $x$-dependence for notational simplicity.) By moving to a
subsequence of $(n_i)_{i=0}^{\infty}$ if necessary and using compactness, we
may assume that each block of \eqref{aux:vector} has a limit for $i\to\infty$.
Repeated application of the induction hypothesis \eqref{induction:hypothesis:p}
then implies that the limit of \eqref{aux:vector} for $i\to\infty$ behaves as
$$ \begin{pmatrix}
\vecw_{k,p-1}(1+O(x^{-l}))\\
\vecw_{k,0}(1+O(x^{-l}))\\
\vdots \\
\vecw_{k,p-2}(1+O(x^{-l}))\\
\vecw_{k,p-1}\\
\varphi(x)
\end{pmatrix},
$$
for $x\to\infty$, where
\begin{equation}\label{phi:def}
\varphi(x):=\left(\lim_{i\to\infty}P_{pn_i+p}(x)P_{pn_i+p-1}^{-1}(x)\right)y_k(x)\vecw_{k,p-1}(x).\end{equation}

Multiplying \eqref{aux:product} on the right with
$P_{pn_i+p-1}^{-1}(x)\vecw_{k,p-1}(x)$ and taking the limit $i\to\infty$, we
get from the above observations that
\begin{equation}\label{cancelation:nullspace}
\begin{pmatrix}
yA^{(0)*} & B^{(0)}-xI_r & y^{-1}A^{(1)} \\
& yA^{(1)*} & \ddots & \ddots  \\
& & \ddots & \ddots & y^{-1}A^{(p-1)}\\
& & & yA^{(p-1)*} & B^{(p-1)}-xI_r & y^{-1}A^{(0)}
\end{pmatrix}
\begin{pmatrix}
\vecw_{k,p-1}\\
\vecw_{k,0}\\
\vdots \\
\vecw_{k,p-2}\\
\vecw_{k,p-1}\\
\varphi(x)
\end{pmatrix}=\begin{pmatrix}O(x^{-l+1})\\ \vdots \\ O(x^{-l+1}) \\ O(x^{-l-1}) \end{pmatrix},
\end{equation}
for $x\to\infty$, where we used that $y\equiv y_k(x)\propto x^{-1}$ for
$x\to\infty$. Taking the last block row of this equation yields
\begin{equation}\label{lastblockrow}
y_k(x)A^{(p-1)*}\vecw_{k,p-2} + (B^{(p-1)}-xI_r)\vecw_{k,p-1}+
y_k^{-1}(x)A^{(0)}\varphi(x)=O(x^{-l-1}).
\end{equation}
On the other hand, by \eqref{vector:w}, \eqref{nullspacevector:bis} and
\eqref{Gyx:bis} (evaluated for the last block row) we have that
$$  y_k^{-1}(x)A^{(0)}\vecw_{k,0} + y_k(x)A^{(p-1)*}\vecw_{k,p-2} +
(B^{(p-1)}-xI_r)\vecw_{k,p-1} = 0.
$$
Subtracting this from \eqref{lastblockrow} we get
$$
A^{(0)}\left(\lim_{i\to\infty}P_{pn_i+p}(x)P_{pn_i+p-1}^{-1}(x)\right)\vecw_{k,p-1}(x)-y_k^{-1}(x)A^{(0)}\vecw_{k,0}(x)=O(x^{-l-1}),
$$
on account of \eqref{phi:def}. The factor $A^{(0)}$ can be skipped from this
equation. Then multiplying on the left with
$y_k(x)\times\left(\lim_{i\to\infty}P_{pn_i+p-1}(x)P_{pn_i+p}^{-1}(x)\right)$
we get
$$
\left(\lim_{i\to\infty}P_{pn_i+p-1}(x)P_{pn_i+p}^{-1}(x)\right)\vecw_{k,0}(x)-y_k(x)\vecw_{k,p-1}(x)=O(x^{-l-3}),
$$
or equivalently
$$ \left(\lim_{i\to\infty}P_{pn_i+p-1}(x)P_{pn_i+p}^{-1}(x)\right)\vecw_{k,0}(x) = y_k(x)\vecw_{k,p-1}(x)(1+O(x^{-l-2})).
$$
We conclude that \eqref{induction:hypothesis:p} holds with $l$ replaced by
$l+2$. This proves the induction step. $\bol$

\subsection{Proof of Proposition~\ref{prop:Cheb12bis}}
\label{subsection:proof:cheb}

{Throughout the proof we will use the notations of
Section~\ref{section:Cheb}.
Recall that the Hermitian symmetry $A=A^*$ implies} the roots $z_k(x)$ to
appear in pairs $\{z_k(x),z_k(x)^{-1}\}$. Both $z_k(x)$ and $z_k(x)^{-1}$
correspond to the same value of $w_k(x)=z_k(x)+z_k(x)^{-1}$ in
\eqref{assumption:Herm2} and therefore to the same null space vector
$\vecv_k(x)$.

Now let $x\in\er$.
For any $w_k(x)=2\cos\theta_k(x)\in(-2,2)$, with $\theta\in (0,\pi)$,
$k=1,\ldots,r$, we have a pair of roots $z_{k_1}(x)=e^{i\theta_k(x)}$ and
$z_{k_2}(x)=e^{-i\theta_k(x)}$, with $\{k_1,k_2\}=\{k,2r-k\}$. Suppose that
$w_k'(x)>0$. Then the Cauchy-Riemann equations show that $z_{k_1}(x)$ lies in
the set $S_-(x)$ in \eqref{Sminus} but not in $S_+(x)$, and vice versa for the
root $z_{k_2}(x)$. The reverse situation occurs if $w_k'(x)<0$.

Fix $x\in\er$ and assume the labeling of roots is such that
$$ \max\{|z_1(x)|,\ldots,|z_{K}(x)|\}<1,\qquad |z_{K+1}(x)|=\ldots =
|z_{r}(x)|=1,
$$
with $K\in\{0,\ldots,r\}$. Taking into account the above observations, we find
from \eqref{Cheb2:cauch1} that
\begin{eqnarray*} && \lim_{\ep\to 0+} F_W(x+\epsilon
i)\\ &=&V(x)\left(\lim_{\ep\to 0+} D(x+\ep i)\right)V^{-1}(x)A^{-1} \\ &=&
V(x)\diag(z_1(x),\ldots,z_K(x),e^{-i\theta_{K+1}(x)\mathrm{sign}\,w_{K+1}'(x)},\ldots,e^{-i\theta_{r}(x)\mathrm{sign}\,w_{r}'(x)})V^{-1}(x)A^{-1}.
\end{eqnarray*}
Similarly
\begin{eqnarray*} && \lim_{\ep\to 0+} F_W(x-\epsilon
i)\\ &=&V(x)\left(\lim_{\ep\to 0+} D(x-\ep i)\right)V^{-1}(x)A^{-1} \\ &=&
V(x)\diag(z_1(x),\ldots,z_K(x),e^{i\theta_{K+1}(x)\mathrm{sign}\,w_{K+1}'(x)},\ldots,e^{i\theta_{r}(x)\mathrm{sign}\,w_{r}'(x)})V^{-1}(x)A^{-1}.
\end{eqnarray*}
Using the Stieltjes inversion principle
$$ \frac{\ud W(x)}{\ud x} = \frac{1}{2\pi i}\lim_{\epsilon\to 0+} \left(F_W(x-\epsilon
i)-F_W(x+\epsilon i)\right),
$$
the desired formula for $\ud W/\ud x$ now follows from a straightforward
calculation. The formula for $\ud X/\ud x$ similarly follows from
\eqref{Cheb1:cauch2}, taking into account the simplifications due to $A=A^*$.
$\bol$

\bigskip

\noindent \textbf{Acknowledgements} The work of the authors is supported by the
SFB TR12 ''Symmetries and Universality in Mesoscopic Systems'', Teilprojekt C2.
The first author is a Postdoctoral Fellow of the Fund for Scientific Research -
Flanders (Belgium). His work is supported in part by the Belgian
Interuniversity Attraction Pole P06/02. The authors would also like to thank
Martina Stein, who typed parts of this manuscript with considerable technical
expertise.

\end{document}